\UseRawInputEncoding
\documentclass{amsart}
\usepackage{url}
\usepackage{amssymb}
\usepackage{verbatim}
\usepackage{hyperref}
\usepackage{color}
\usepackage{mathrsfs}

\begin{document}

\newtheorem{theorem}{Theorem}    
\newtheorem{proposition}[theorem]{Proposition}
\newtheorem{conjecture}[theorem]{Conjecture}
\def\theconjecture{\unskip}
\newtheorem{corollary}[theorem]{Corollary}
\newtheorem{lemma}[theorem]{Lemma}
\newtheorem{sublemma}[theorem]{Sublemma}
\newtheorem{observation}[theorem]{Observation}
\newtheorem{remark}[theorem]{Remark}
\newtheorem{definition}[theorem]{Definition}
\theoremstyle{definition}
\newtheorem{notation}[theorem]{Notation}
\newtheorem{question}[theorem]{Question}
\newtheorem{questions}[theorem]{Questions}
\newtheorem{example}[theorem]{Example}
\newtheorem{problem}[theorem]{Problem}
\newtheorem{exercise}[theorem]{Exercise}

\numberwithin{theorem}{section} \numberwithin{theorem}{section}
\numberwithin{equation}{section}

\def\earrow{{\mathbf e}}
\def\rarrow{{\mathbf r}}
\def\uarrow{{\mathbf u}}
\def\varrow{{\mathbf V}}
\def\tpar{T_{\rm par}}
\def\apar{A_{\rm par}}

\def\reals{{\mathbb R}}
\def\torus{{\mathbb T}}
\def\heis{{\mathbb H}}
\def\integers{{\mathbb Z}}
\def\naturals{{\mathbb N}}
\def\complex{{\mathbb C}\/}
\def\distance{\operatorname{distance}\,}
\def\support{\operatorname{support}\,}
\def\dist{\operatorname{dist}\,}
\def\Span{\operatorname{span}\,}
\def\degree{\operatorname{degree}\,}
\def\kernel{\operatorname{kernel}\,}
\def\dim{\operatorname{dim}\,}
\def\codim{\operatorname{codim}}
\def\trace{\operatorname{trace\,}}
\def\Span{\operatorname{span}\,}
\def\dimension{\operatorname{dimension}\,}
\def\codimension{\operatorname{codimension}\,}
\def\nullspace{\scriptk}
\def\kernel{\operatorname{Ker}}
\def\ZZ{ {\mathbb Z} }
\def\p{\partial}
\def\rp{{ ^{-1} }}
\def\Re{\operatorname{Re\,} }
\def\Im{\operatorname{Im\,} }
\def\ov{\overline}
\def\eps{\varepsilon}
\def\lt{L^2}
\def\diver{\operatorname{div}}
\def\curl{\operatorname{curl}}
\def\etta{\eta}
\newcommand{\norm}[1]{ \|  #1 \|}
\def\expect{\mathbb E}
\def\bull{$\bullet$\ }
\def\C{\mathbb{C}}
\def\R{\mathbb{R}}
\def\Rn{{\mathbb{R}^n}}
\def\Sn{{{S}^{n-1}}}
\def\M{\mathbb{M}}
\def\N{\mathbb{N}}
\def\Q{{\mathbb{Q}}}
\def\Z{\mathbb{Z}}
\def\F{\mathcal{F}}
\def\L{\mathcal{L}}
\def\S{\mathcal{S}}
\def\supp{\operatorname{supp}}
\def\dist{\operatorname{dist}}
\def\essi{\operatornamewithlimits{ess\,inf}}
\def\esss{\operatornamewithlimits{ess\,sup}}
\def\xone{x_1}
\def\xtwo{x_2}
\def\xq{x_2+x_1^2}
\newcommand{\abr}[1]{ \langle  #1 \rangle}

\newcommand{\Norm}[1]{ \left\|  #1 \right\| }
\newcommand{\set}[1]{ \left\{ #1 \right\} }
\def\one{\mathbf 1}
\def\whole{\mathbf V}
\newcommand{\modulo}[2]{[#1]_{#2}}

\def\scriptf{{\mathcal F}}
\def\scriptg{{\mathcal G}}
\def\scriptm{{\mathcal M}}
\def\scriptb{{\mathcal B}}
\def\scriptc{{\mathcal C}}
\def\scriptt{{\mathcal T}}
\def\scripti{{\mathcal I}}
\def\scripte{{\mathcal E}}
\def\scriptv{{\mathcal V}}
\def\scriptw{{\mathcal W}}
\def\scriptu{{\mathcal U}}
\def\scriptS{{\mathcal S}}
\def\scripta{{\mathcal A}}
\def\scriptr{{\mathcal R}}
\def\scripto{{\mathcal O}}
\def\scripth{{\mathcal H}}
\def\scriptd{{\mathcal D}}
\def\scriptl{{\mathcal L}}
\def\scriptn{{\mathcal N}}
\def\scriptp{{\mathcal P}}
\def\scriptk{{\mathcal K}}
\def\frakv{{\mathfrak V}}

%
\newtheorem*{remark0}{\indent\sc Remark}
%
\renewcommand{\proofname}{\indent\sc Proof.} 

\title[Some sharp bounds for Hardy type operators]
{Some sharp bounds for Hardy type operators on mixed radial-angular type function spaces}
\author{Ronghui Liu$^{*,1,2}$}
\author{Yanqi Yang$^{1,2}$}
\author{Shuangping Tao$^{1}$}

\renewcommand{\thefootnote}{}
\footnotetext[1]{*Corresponding author: Ronghui Liu}
\footnotetext[1]{\textcolor{black}{2020} \textit{Mathematics Subject Classification}. Primary 42B35; Secondary 26D10, 46E30, 26D15.}


%
\thanks{ Ronghui Liu, email:rhliu@nwnu.edu.cn.   Yanqi Yang, email: yangyq@nwnu.edu.cn  \\ Shuangping Tao, email: taosp@nwnu.edu.cn.  }
\thanks{ $^{1}$ College of Mathematics and Statistics, Northwest Normal University, Lanzhou, 730070, People's Republic
of China}
\thanks{$^{2}$ Hubei Key Laboratory of Applied Mathematics, Hubei University, Wuhan, 430062, People's Republic
of China}
\thanks{The first author was supported by the Doctoral Scientific Research Foundation of Northwest Normal University (No. 202203101202), the Young Teachers Scientific Research Ability Promotion Project of Northwest Normal University (No. NWNU-LKQN2023-15)
and the Open Foundation of Hubei Key Laboratory of Applied Mathematics (Hubei University) (No. HBAM202303). The second author was supported by the Open Foundation of Hubei Key Laboratory of Applied Mathematics (Hubei University) (No. HBAM202205). The third author was supported by the NNSF of China (No. 12361018.)}


\maketitle
\begin{abstract}
In this paper, we are devoted to studying some sharp bounds for Hardy type operators on mixed radial-angular type function spaces. In addition, we  will  establish the sharp weak-type estimates for the fractional Hardy operator and its conjugate operator, respectively.
\end{abstract}

\section{Some classical conclusions of Hardy type operators }
The Hardy operator, as the most fundamental averaging operator, is defined by
 $$\mathrm{H}f(x)= \frac{1}{x}\int_{0}^{x}f(t)dt,$$
 where the function $f$ is  a nonnegative  integrable function on $\mathbb{R}^{+}$ and $x>0$. A celebrated integral inequality, due to Hardy \cite{GH}, states that
  $$\|\mathrm{H}f\|_{L^{p}(\mathbb{R}^{+})}\leq\frac{p}{p-1}\|f\|_{L^{p}(\mathbb{R}^{+})}$$
  holds for $1<p<\infty$, and the constant $\frac{p}{p-1}$ is the best possible.

For the multidimensional case $n\geq2 $, generally speaking, there exist two different definitions. One is the rectangle averaging operator, that is,
$$\mathfrak{H}f(x)= \frac{1}{x_1\cdots x_m}\int_{0}^{x_1}\cdots\int_{0}^{x_m}f(t_1,\ldots,t_m)dt_m\cdots dt_1,$$
where the function $f$ is  a nonnegative  integrable function on $(0,\infty)^n$, and
$$\|\mathfrak{H}\|_{L^p((0,\infty)^n)\rightarrow L^p((0,\infty)^n)}=\big(\frac{p}{p-1}\big)^n,\quad p>1,$$
it is easy to see that its norm  depends on the dimensions.

Another version  is the $n$-dimensional spherical  averaging operator, which  was  introduced by Faris in \cite{CG} as follows:
$$\mathcal{H}f(x)= \frac{1}{\nu_n|x|^n}\int_{|y|\leq|x|}f(y)dy, \quad x\in\mathbb{R}^n\setminus \{0\},$$
where $\nu_n=\frac{\pi^{n/2}}{\Gamma(1+n/2)}$ is the volume of the unit ball in $\mathbb{R}^n$. The norm of $\mathcal{H}$ on $L^{p}(\mathbb{R}^{n})$ was evaluated  as
$$\|\mathcal{H}\|_{L^p(\mathbb{R}^n)\rightarrow L^p(\mathbb{R}^n)}=\frac{p}{p-1}, \quad p>1.$$

However, it is not hard to see that its norm is different from the rectangle averaging operator and found to be equal to that of the $1$-dimensional averaging operator, that is to say, does not depend on the dimension of the space.

Suppose $\psi:[0,1]\rightarrow[0,\infty)$ is a function. For a measurable complex valued function $f$ on $\mathbb R^{n}$, Xiao in \cite{Xiao} defined the weighted Hardy-Littlewood average $U_{\psi}(f)$ and the weighted Ces\`{a}ro average $V_{\psi}(f)$ as
$$U_{\psi}(f)(x)=\int_{0}^{1}f(tx)\psi(t)dt, \quad V_{\psi}(f)(x)=\int_{0}^{1}f(x/t)t^{-n}\psi(t)dt.$$
It is easy to see that if $\psi=1$ and $n=1$, then
$$U_{\psi}=\mathrm{H},\quad V_{\psi}=\int_{x}^{\infty}\frac{f(y)}{y}dy, \quad x>0,$$
and $$U_{\psi}(f)(x)+U_{\psi}(f)(x)= \frac{1}{x}\int_{0}^{x}f(y)dy+\int_{x}^{\infty}\frac{f(y)}{y}dy, \quad x>0$$
becomes the Calder\'{o}n maximal operator.

In \cite{Xiao}, the author obtained the following results for the operators $U_{\psi}$ and  $V_{\psi}$.

{\bf Theorem A}  Let $\psi:[0,1]\rightarrow[0,\infty)$ be a function and let $p\in[1,\infty)$. Then

(i) $U_{\psi}(f): L^p(\mathbb R^{n})\rightarrow L^p(\mathbb R^{n})$ exists as a bounded operator if and only if
$$\int_{0}^{1}t^{-n/p}\psi(t)dt<\infty.$$

Moreover, the operator norm of $U_{\psi}$ on $L^p(\mathbb R^{n})$ is given by
$$\|U_{\psi}\|_{L^p(\mathbb R^{n})\rightarrow L^p(\mathbb R^{n})}=\int_{0}^{1}t^{-n/p}\psi(t)dt.$$

(ii) $V_{\psi}(f): L^p(\mathbb R^{n})\rightarrow L^p(\mathbb R^{n})$ exists as a bounded operator if and only if
$$\int_{0}^{1}t^{-n(1-1/p)}\psi(t)dt<\infty.$$

Moreover, the operator norm of $V_{\psi}$ on $L^p(\mathbb R^{n})$ is given by
$$\|V_{\psi}\|_{L^p(\mathbb R^{n})\rightarrow L^p(\mathbb R^{n})}=\int_{0}^{1}t^{-n(1-1/p)}\psi(t)dt.$$

 In 2013, Lu et al. \cite{LYZ} defined  the following Hardy type operator:
$$\mathcal{H}_{m}(f)(x):=\Big(\prod\limits_{i=1}^{m}\frac{1}{|B(0,|x_i|)|}\Big)\int_{|y_1|<|x_1|}\cdots\int_{|y_m|<|x_m|}f(y_1,\ldots,y_m)dy_m\cdots{dy_1},$$
where  $f$ is a  nonnegative measurable function on $\mathbb R^{n_1} \times\mathbb R^{n_2}\times\cdots\times\mathbb R^{n_m}$, $ m\in\mathbb N$, $n_i\in\mathbb N$, $x=(x_1,x_2,\ldots,x_m)\in\mathbb R^{n_1} \times\mathbb R^{n_2}\times\cdots\times\mathbb R^{n_m}$, $x_i\in\mathbb R^{n_i}$ and $\prod\limits_{i=1}^{m}|x_i|\neq0$. Furthermore, the corresponding operator norm on the Lebesgue product spaces with power weights was obtained by some classical techniques.

{\bf Theorem B} (\cite{LYZ})
Let $1<q<\infty, m\in\mathbb N$, $n_i\in\mathbb N$, $x_i\in\mathbb R^{n_i}$, $i=1,\ldots,m$. If $f\in L^q(\mathbb R^{n_1} \times\mathbb R^{n_2}\times\cdots\times\mathbb R^{n_m},|x|^{\vec\alpha})$, where $|x|^{\vec\alpha}:=|x|^{\alpha_1}\times|x|^{\alpha_2}\cdots\times|x|^{\alpha_m}$ and $\alpha_i<(q-1)n_i$, then the Hardy type operator $\mathcal{H}_{m}$ is bounded on $L^q(\mathbb R^{n_1} \times\mathbb R^{n_2}\times\cdots\times\mathbb R^{n_m},|x|^{\vec\alpha})$, moreover, the norm of $\mathcal{H}_{m}$ can be obtained as follows:
$$\|\mathcal{H}_{m}\|_{{L^q(\mathbb R^{n_1} \times\mathbb R^{n_2}\times\cdots\times\mathbb R^{n_m},|x|^{\vec\alpha})}\rightarrow L^q(\mathbb R^{n_1} \times\mathbb R^{n_2}\times\cdots\times\mathbb R^{n_m},|x|^{\vec\alpha})}=\prod\limits_{i=1}^{m}\frac{q}{q-1-\alpha_i/n_i}.$$

Subsequently, Fu et al.  in \cite{FLLZ}  established sharp bound for $m$-linear Hardy operator, let $m\in\mathbb{N}, f_1,f_2,\ldots,f_m$ be nonnegative locally functions on $\mathbb R^{n}$. The $m$-linear Hardy operator is defined by
$$\mathcal{H}^{m}(f)(x):=\frac{1}{\nu_{mn}}\frac{1}{|x|^{mn}}\int_{|(y_1,\cdots,y_m)|<|x|}f(y_1)\cdots f(y_m)dy_1\cdots{dy_m},\quad x\in\mathbb{R}^n\setminus \{0\}.$$

{\bf Theorem C} (\cite{FLLZ})
Let $m\in\mathbb N$,  $f_i\in L^{p_i}(|x|^{\frac{\alpha_ip_i}{p}}dx)$, $1<p_i<\infty$, $1\leq p<\infty$, $i=1,\ldots,m$, $\frac{1}{p}=\frac{1}{p_1}+\cdots+\frac{1}{p_m}$, $\alpha_i<(1-1/p_i)pn$ and $\alpha=\alpha_1+\cdots+\alpha_m$. Then $m$-linear Hardy operator $\mathcal{H}^{m}$ maps the product of weighted Lebesgue spaces  $L^{p_1}(|x|^{\frac{\alpha_1p_1}{p}}dx)\times\cdots\times L^{p_m}(|x|^{\frac{\alpha_mp_m}{p}}dx)$ to  $L^{p}(|x|^{\alpha}dx)$ with norm equal to the constant
$$\frac{\omega_n^m}{\omega_{mn}}\frac{pmn}{{pmn-n-\alpha}}\frac{1}{2^m-1}\frac{\prod_{i=1}^{m}\Gamma(\frac{n}{2}(1-\frac{1}{p_i}-\frac{\alpha_i}{p_i}))}{\Gamma(\frac{n}{2}(1-\frac{1}{p}-\frac{\alpha}{p}))}.$$

\begin{definition}\label{de1.1}
Let $1<p<\infty$. A function $f\in L^p_{loc}(\mathbb R^{n})$ is said to belong to the spaces ${CMO^{p}(\mathbb R^n)}$, if
$$\|f\|_{{CMO^{p}(\mathbb R^n)}}=\sup\limits_{r>0}\Big(\frac{1}{|B(0,r)|}\int_{B(0,r)}|f(x)-f_{B}|^pdx\Big)^{1/p}<\infty,$$
where $$f_{B}=\frac{1}{|B(0,r)|}\int_{B(0,r)}f(x)dx.$$
\end{definition}

In 2000, Alvarez, Guzm\'{a}n-Partida and Lakey \cite{AGL} introduced $\lambda$-central bounded mean oscillation spaces  and  $\lambda$-central Morrey spaces as follows.

\begin{definition}\label{de1.2} Let $1<q<\infty$ and $-1/q<\lambda<1/n$. A function $f\in L^q_{loc}(\mathbb R^{n})$ is said to belong to the $\lambda$-central bounded mean oscillation spaces $CMO^{q,\lambda}(\mathbb R^{n})$ if
$$\|f\|_{{CMO^{q,\lambda}(\mathbb R^n)}}=\sup\limits_{r>0}\Big(\frac{1}{|B(0,r)|^{1+\lambda q}}\int_{B(0,r)}|f(x)-f_{B}|^qdx\Big)^{1/q}<\infty.$$
\end{definition}

\begin{definition}\label{de1.3} Let $1<q<\infty$ and $\lambda\in\mathbb{R}$. A function $f\in L^q_{loc}(\mathbb R^{n})$ is said to belong to the $\lambda$-central Morrey spaces $\mathcal{B}^{q,\lambda}(\mathbb R^{n})$ if
$$\|f\|_{\mathcal{B}^{q,\lambda}(\mathbb R^{n})}=\sup\limits_{r>0}\Big(\frac{1}{|B(0,r)|^{1+\lambda q}}\int_{B(0,r)}|f(x)|^qdx\Big)^{1/q}<\infty.$$
\end{definition}

In \cite{FL1}, Fu et al. proved the following sharp estimates for the $n$-dimensional spherical  averaging operator $\mathcal{H}$.

{\bf Theorem D}  Let $1<q<\infty$ and $-1/q\leq\lambda\leq0$. Then
$$\|\mathcal{H}f\|_{\mathcal{B}^{q,\lambda}(\mathbb R^{n})}\leq\frac{1}{1+\lambda}\|f\|_{\mathcal{B}^{q,\lambda}(\mathbb R^{n})}.$$
Moreover,
$$\|\mathcal{H}\|_{\mathcal{B}^{q,\lambda}(\mathbb R^{n})\rightarrow \mathcal{B}^{q,\lambda}(\mathbb R^{n})}=\frac{1}{1+\lambda}.$$

{\bf Theorem E}  Let $1<q<\infty$ and $-1/q<\lambda<1/n$. Then
$$\|\mathcal{H}f\|_{{CMO^{q,\lambda}(\mathbb R^n)}}\leq\frac{1}{1+\lambda}\|f\|_{{CMO^{q,\lambda}(\mathbb R^n)}}.$$
Moreover,
$$\|\mathcal{H}\|_{{CMO^{q,\lambda}(\mathbb R^n)}\rightarrow {CMO^{q,\lambda}(\mathbb R^n)}}=\frac{1}{1+\lambda}.$$
Furthermore,
$$\|\mathcal{H}\|_{{CMO^{q}(\mathbb R^n)}\rightarrow {CMO^{q}(\mathbb R^n)}}=1.$$

In \cite{FLLZ}, Fu et al. established the following sharp bound for the $m$-linear Hardy operator $\mathcal{H}^{m}$.

{\bf Theorem F} (\cite{FLLZ})
Let $m\in\mathbb N$,  $f_i\in \mathcal{B}^{p_i,\lambda_i}(\mathbb R^{n})$, $1<p_i<\infty$, $1<p<\infty$, $\frac{1}{p}=\frac{1}{p_1}+\cdots+\frac{1}{p_m}$, $-1/p_i\leq\lambda_i<0$ $i=1,\ldots,m$, and $\lambda=\lambda_1+\cdots+\lambda_m$. Then $m$-linear Hardy operator $\mathcal{H}^{m}$ maps $\mathcal{B}^{p_1,\lambda_1}(\mathbb R^{n})\times\cdots\times \mathcal{B}^{p_2,\lambda_2}(\mathbb R^{n})$ to  $\mathcal{B}^{p,\lambda}(\mathbb R^{n})$ with norm equal to the constant
$$\frac{\omega_n^m}{\omega_{mn}}\frac{m}{\lambda+m}\frac{1}{2^m-1}\frac{\prod_{i=1}^{m}\Gamma(\frac{n}{2}(1+\lambda_i))}{\Gamma(\frac{n}{2}(m+\lambda))}.$$

Let $B_k=\{x\in\mathbb R^n : |x|\leq 2^k \}$, $C_k=B_k\backslash B_{k-1}$ and $\chi_k=\chi_{C_k}$ for $k\in\mathbb Z$, where $\chi_{E}$ is the characteristic function of set $E$.

\begin{definition}\label{de1.5} (\cite{LD}) Let $\alpha\in\mathbb R$, $0<p,q<\infty$. The homogeneous Herz spaces $\dot{K}^{\alpha,q}_{p}(\mathbb R^{n})$
are defined by
$$\dot{K}^{\alpha,q}_{p}(\mathbb R^{n})=\Big\{f\in L^p_{loc}(\mathbb R^{n} \setminus{\{0\}}): \|f\|_{{\dot{K}}^{\alpha,q}_{p}(\mathbb R^{n})}<\infty\Big\},$$
where
$$\|f\|_{{\dot{K}}^{\alpha,q}_{p}(\mathbb R^{n})}=\Big\{\sum\limits_{k=-\infty}^{\infty}2^{k\alpha q} \|f\chi_k\|^q_{L^p(\mathbb R^{n})}\Big\}^{1/q}.$$
\end{definition}

{\bf Theorem G} (\cite{FL}) Let $\alpha\in\mathbb{R}$, $1<p,q<\infty$. Then the Hardy type operator $U_{\psi}$ is bounded from
$\dot{K}^{\alpha,q}_{p}(\mathbb R^{n})$ to itself if
$$\int_{0}^{1}t^{-\alpha-n/p}\psi(t)dt.$$
Moreover, the operator norm of $U_{\psi}$ satisfies
$$\|U_{\psi}\|_{\dot{K}^{\alpha,q}_{p}(\mathbb R^{n})\rightarrow \dot{K}^{\alpha,q}_{p}(\mathbb R^{n})}\simeq\int_{0}^{1}t^{-\alpha-n/p}\psi(t)dt,$$
and the operator norm of $V_{\psi}$ satisfies
$$\|V_{\psi}\|_{\dot{K}^{\alpha,q}_{p}(\mathbb R^{n})\rightarrow \dot{K}^{\alpha,q}_{p}(\mathbb R^{n})}\simeq\int_{0}^{1}t^{-\alpha-n+n/p}\psi(t)dt.$$

\begin{definition}\label{de1.5} (\cite{LX}) Let $\alpha\in\mathbb R$, $0<p, q<\infty$ and $\lambda\geq0$. The homogeneous Morrey-Herz spaces $M\dot{K}^{\alpha,\lambda}_{p,q}(\mathbb R^{n})$
are defined by
$$M\dot{K}^{\alpha,\lambda}_{p,q}(\mathbb R^{n})=\Big\{f\in L^p_{loc}(\mathbb R^{n} \setminus{\{0\}}): \|f\|_{M{\dot{K}}^{\alpha,\lambda}_{p,q}(\mathbb R^{n})}<\infty\Big\},$$
where
$$\|f\|_{{M\dot{K}}^{\alpha,\lambda}_{p,q}(\mathbb R^{n})}=\sup\limits_{k_0\in\mathbb{Z}}2^{-k_0\lambda}\Big\{\sum\limits_{k=-\infty}^{k_0}2^{k\alpha q} \|f\chi_k\|^q_{L^p(\mathbb R^{n})}\Big\}^{1/q}.$$
\end{definition}

{\bf Theorem H} (\cite{FL}) Let $\alpha\in\mathbb{R}$, $1<p,q<\infty$ and $\lambda>0$. Then the Hardy type operator $U_{\psi}$ is bounded from
$M\dot{K}^{\alpha,\lambda}_{p,q}(\mathbb R^{n})$ to itself if
$$\int_{0}^{1}t^{-\alpha-n/p+\lambda}\psi(t)dt.$$
Moreover, the operator norm of $U_{\psi}$ satisfies
$$\|U_{\psi}\|_{M\dot{K}^{\alpha,\lambda}_{p,q}(\mathbb R^{n})\rightarrow M\dot{K}^{\alpha,\lambda}_{p,q}(\mathbb R^{n})}\simeq\int_{0}^{1}t^{-\alpha-n/p+\lambda}\psi(t)dt,$$
and the operator norm of $V_{\psi}$ satisfies
$$\|V_{\psi}\|_{M\dot{K}^{\alpha,\lambda}_{p,q}(\mathbb R^{n})\rightarrow M\dot{K}^{\alpha,\lambda}_{p,q}(\mathbb R^{n})}\simeq\int_{0}^{1}t^{\alpha-n+n/p-\lambda}\psi(t)dt.$$
\section{Some mixed radial-angular type function spaces}

Let $S^{n-1}$ be the unit sphere in $\mathbb R^n$, $n\geq 2$, with  Lebesgue measure $d\sigma=d\sigma(\cdot)$. For any $f\in L^p(\mathbb R^n)$, $1\leq p<\infty$, applying the spherical coordinate formula, we write
\begin{align*}
\|f\|_{L^p(\mathbb R^n)}&=\Big(\int_{0}^{\infty}\int_{S^{n-1}}|f(r\theta)|^pd\sigma(\theta)r^{n-1}dr\Big)^{1/p}\\
&=\Big(\int_{0}^{\infty}\|f(r\cdot)\|^p_{L^p(S^{n-1})}r^{n-1}dr\Big)^{1/p}.
\end{align*}
Therefore, from the perspective of radial and angular integrabilities, Lebesgue norms can be regarded as certain special norms with the same integrabilities in the radial and angular directions.  Motivated by this form,  we naturally consider the case of Lebesgue norms with different integrabilities in the radial and angular directions, namely,
 $$\|f\|_{L^{p}_{rad}L^{\tilde{p}}_{ang}(\mathbb R^n)}:=\bigg(\int_{0}^{\infty}\|f(r\cdot)\|^p_{L^{\tilde{p}}(S^{n-1})}r^{n-1}dr\bigg)^{{1}/{p}},\quad  1\leq p,
\tilde{p}\leq\infty,$$
and when $p=\infty$ or $\tilde{p}=\infty$, we just need to make the usual modifications in the
above definition, but we do not use these cases in the current work.

In addition, the mixed radial-angular space ${L^{p}_{rad}L^{\tilde{p}}_{ang}(\mathbb R^n)}$, as a formal extension of the Lebesgue spaces $L^p(\mathbb R^n)$, was introduced to study of regularity and some important estimates, such as angular regularity and Strichartz estimates (see \cite{CL,DL2,St,Tao} etc.). Recently, the first author et al. also established the boundedness of some classical operators with rough kernels on mixed radial-angular spaces in \cite{LLW,LLW2,LW,LRW,LTW,LTW2}.

In 2018, Duoandikoetxea et al. \cite{DR} defined the following  weighted mixed-norm, and established some  weighted mixed-norm inequalities by the usual extrapolation methods.

\begin{definition}\label{de2.1} Let $\omega$ be a nonnegative locally integrable function on $\mathbb R^+$, for any $1\leq p,\tilde{p}<\infty$, the weighted mixed radial-angular spaces  ${L^{p}_{rad}L^{\tilde{p}}_{ang}(\mathbb R^n, \omega)}$ are defined by

 $${{L^{p}_{rad}L^{\tilde{p}}_{ang}(\mathbb R^n, \omega)}}=\big\{f:\,\|f\|_{L^{p}_{rad}L^{\tilde{p}}_{ang}(\mathbb R^n, \omega)} <\infty\big\},$$
 where
 $$\|f\|_{L^{p}_{rad}L^{\tilde{p}}_{ang}(\mathbb R^n, \omega)}:=\bigg(\int_{0}^{\infty}\|f(r\cdot)\|^p_{L^{\tilde{p}}(S^{n-1})}r^{n-1}\omega(r)dr\bigg)^{{1}/{p}}.$$
\end{definition}

Whereafter, the first author and third author in  \cite{LT} proved the following conclusion.

\begin{theorem}\label{th2.1}
Let $1<p_i, q_i, \tilde{q}_{i}<\infty$, $n_i\in\mathbb N$, $x_i\in\mathbb R^{n_i}$,  $r_i\in (0,\infty)$, $i=1,\ldots,m$. If $f\in L^{\vec{p}}_{rad}L^{\vec{\tilde{q}}}_{ang}(\mathbb R^{\vec{n}}, r^{\vec{\alpha}})$, where $r^{\vec{\alpha}}=r_1^{{\alpha}_1}\times\cdots \times r_2^{{\alpha}_m}$ and ${\alpha}_i<(p_i-1)n_i$. Then the Hardy type operator $\mathcal{H}_m$ is bounded from $L^{\vec{p}}_{rad}L^{\vec{\tilde{q}}}_{ang}(\mathbb R^{\vec{n}}, r^{\vec{\alpha}})$ to  $L^{\vec{p}}_{rad}L^{\vec{q}}_{ang}(\mathbb R^{\vec{n}}, r^{\vec{\alpha}})$, moreover, the norm of $\mathcal{H}_m$ can be obtained as follows:

$$\|\mathcal{H}_m\|_{L^{\vec{p}}_{rad}L^{\vec{q}}_{ang}(\mathbb R^{\vec{n}}, r^{\vec{\alpha}})\rightarrow L^{\vec{p}}_{rad}L^{\vec{\tilde{q}}}_{ang}(\mathbb R^{\vec{n}}, r^{\vec{\alpha}})}=\prod\limits_{i=1}^{m}{{\omega}_i}^{\frac{1}{\vec{q}_{i}}-\frac{1}{\vec{\tilde{q}}_{i}}}\bigg(\frac{p_i}{p_i-1-\frac{{\alpha}_i}{n_j}}\bigg). $$
\end{theorem}

Inspired by the mixed radial-angular integrabilities, the first author et al. in \cite{LTW1,LTW2} introduced some other mixed radial-angular type function spaces as follows.

\begin{definition}\label{de2.2}
Let $1<p_1, p_2<\infty$. A function $f\in L^{p_2}_{\rm  rad,{loc}}L^{p_1}_{\rm ang}((0,\infty)\times{S}^{n-1})$ is said to belong to the mixed radial-angular homogeneous CMO  spaces ${CMOL^{p_2}_{\rm rad}L^{p_1}_{\rm ang}(\mathbb R^n)}$, if
\begin{small}
\begin{equation*}
\|f\|_{{CMOL^{p_2}_{\rm rad}L^{p_1}_{\rm ang}(\mathbb R^n)}}=\sup\limits_{r>0}\Big(\frac{1}{\nu_nr^n}\int_{0}^{r}\Big(\int_{{S}^{n-1}}|f(\rho\theta)-f_{B}|^{p_1}d\sigma(\theta)\Big)^{p_2/p_1}\rho^{n-1}d\rho\Big)^{1/p_2}<\infty,
\end{equation*}
\end{small}
where $f\in L^{p_2}_{\rm rad,{loc}}L^{p_1}_{\rm ang}((0,\infty)\times{S}^{n-1})$ means that the function $f$ defined on $(0,r_0)$ \\$\times{S}^{n-1}\subset(0,\infty)\times{S}^{n-1}$ satisfies
$\|f\|_{L^{p_2}_{\rm rad}L^{p_1}_{\rm ang}((0,r_0)\times{S}^{n-1})}<\infty$ for any $r_0\in(0,\infty)$.
\end{definition}

\begin{remark}\label{re2.3}
(1) if $1<p_1=p_2=p<\infty$, then
$${CMOL^{p_2}_{\rm rad}L^{p_1}_{\rm ang}(\mathbb R^n)}={CMO^{p}(\mathbb R^n)}.$$
(2) if $1<p_1\leq \tilde{p_1}<\infty$, then by the H\"{o}lder inequality, we have

$${CMOL^{p_2}_{\rm rad}L^{\tilde{p_1}}_{\rm ang}(\mathbb R^n)}\subset{CMOL^{p_2}_{\rm rad}L^{p_1}_{\rm ang}(\mathbb R^n)}.$$
(3) if $1<p_2\leq \tilde{p_2}<\infty$, then by the H\"{o}lder inequality, we also have

$${CMOL^{\tilde{p_2}}_{\rm rad}L^{{p_1}}_{\rm ang}(\mathbb R^n)}\subset{CMOL^{p_2}_{\rm rad}L^{p_1}_{\rm ang}(\mathbb R^n)}.$$
\end{remark}

\begin{definition}\label{de2.4}
Let $1<p_1, p_2<\infty$  and $-1/p_2<\lambda<1/n$. A function $f\in L^{p_2}_{\rm rad,{loc}}L^{p_1}_{\rm ang}((0,\infty)\times{S}^{n-1})$ is said to belong to the mixed radial-angular homogeneous $\lambda$-central bounded mean oscillation  spaces ${CMOL^{p_2, \lambda}_{\rm rad}L^{p_1}_{\rm ang}(\mathbb R^n)}$, if
\begin{align*}
&\|f\|_{{CMOL^{p_2,\lambda}_{\rm rad}L^{p_1}_{\rm ang}(\mathbb R^n)}}\\
&=\sup\limits_{r>0}\Big(\frac{1}{\nu_nr^{n+n\lambda p_2}}\int_{0}^{r}\Big(\int_{{S}^{n-1}}|f(\rho\theta)-f_{B}|^{p_1}d\sigma(\theta)\Big)^{p_2/p_1}\rho^{n-1}d\rho\Big)^{1/p_2}<\infty.
\end{align*}
\end{definition}

\begin{remark}\label{re2.5}
(1) if $1<p_1=p_2=p<\infty$ and $-1/p_2<\lambda<1/n$, then
$${CMOL^{p_2,\lambda}_{\rm rad}L^{p_1}_{\rm ang}(\mathbb R^n)}={CMO^{p,\lambda}(\mathbb R^n)}.$$
(2) if $1<p_1\leq \tilde{p_1}<\infty$, and $-1/p_2<\lambda<1/n$, then  we have
$${CMOL^{p_2,\lambda}_{\rm rad}L^{\tilde{p_1}}_{\rm ang}(\mathbb R^n)}\subset{CMOL^{p_2,\lambda}_{\rm rad}L^{p_1}_{\rm ang}(\mathbb R^n)}.$$
(3) if $1<p_2\leq \tilde{p_2}<\infty$ and $-1/p_2<\lambda<1/n$, then we also have
$${CMOL^{\tilde{p_2},\lambda}_{\rm rad}L^{{p_1}}_{\rm ang}(\mathbb R^n)}\subset{CMOL^{p_2,\lambda}_{\rm rad}L^{p_1}_{\rm ang}(\mathbb R^n)}.$$
(4) for $1<p_1, p_2<\infty$ and $-1/p_2<\lambda<1/n$. Then $f\in{CMOL^{p_2,\lambda}_{\rm rad}L^{p_1}_{\rm ang}(\mathbb R^n)}$ $\Longleftrightarrow$ there exists the constant $a_B$ related to the ball $B(0,r)$ such that
 $$\int_{0}^{r}\Big(\int_{{S}^{n-1}}|f(\rho\theta)-a_B|^{p_1}d\sigma(\theta)\Big)^{p_2/p_1}\rho^{n-1}d\rho\lesssim r^{n+n\lambda p_2}.$$
(5) for $1<p_1, p_2<\infty$ and $-1/p_2<\lambda<1/n$. Then $f\in{CMOL^{p_2,\lambda}_{\rm rad}L^{p_1}_{\rm ang}(\mathbb R^n)}$ $\Longleftrightarrow$
 $$\sup\limits_{r>0}\inf\limits_{c\in\mathbb{C}}\Big(\frac{1}{\nu_nr^{n+n\lambda p_2}}\int_{0}^{r}\Big(\int_{{S}^{n-1}}|f(\rho\theta)-c|^{p_1}d\sigma(\theta)\Big)^{p_2/p_1}\rho^{n-1}d\rho\Big)^{1/p_2}<\infty.$$
\end{remark}

\begin{definition}\label{de2.6}
Let $1<p_1, p_2<\infty$  and $\lambda\in \mathbb{R}$. A function $f\in L^{p_2}_{\rm rad,{loc}}L^{p_1}_{\rm ang}((0,\infty)$ $\times{S}^{n-1})$ is said to belong to the  mixed radial-angular homogeneous $\lambda$-central Morrey spaces ${\mathcal{B}L^{p_2,\lambda}_{\rm rad}L^{p_1}_{\rm ang}(\mathbb R^n)}$, if
\begin{small}
$$\|f\|_{{\mathcal{B}L^{p_2,\lambda}_{\rm rad}L^{p_1}_{\rm ang}(\mathbb R^n)}}=\sup\limits_{r>0}\Big(\frac{1}{\nu_nr^{n+n\lambda p_2}}\int_{0}^{r}\Big(\int_{{S}^{n-1}}|f(\rho\theta)|^{p_1}d\sigma(\theta)\Big)^{p_2/p_1}\rho^{n-1}d\rho\Big)^{1/p_2}<\infty.$$
\end{small}
\end{definition}

\begin{definition}\label{de2.7} Let $\alpha\in\mathbb R$, $0<p_1,p_2, q<\infty$. The mixed radial-angular homogeneous Herz spaces $\dot{K}^{\alpha,q}_{L^{p_2}_{\rm rad}L^{p_1}_{\rm ang}}(\mathbb R^{n})$
are defined by
$$\dot{K}^{\alpha,q}_{L^{p_2}_{\rm rad}L^{p_1}_{\rm ang}}(\mathbb R^{n})=\Big\{f\in L^{p_2}_{\rm rad,{loc}}L^{p_1}_{\rm ang}((0,\infty)\times{S}^{n-1}) \setminus{\{0\}}): \|f\|_{\dot{K}^{\alpha,q}_{L^{p_2}_{\rm rad}L^{p_1}_{\rm ang}}(\mathbb R^{n})}<\infty\Big\},$$
where
$$\|f\|_{{\dot{K}}^{\alpha,q}_{L^{p_2}_{\rm rad}L^{p_1}_{\rm ang}}(\mathbb R^{n})}=\Big\{\sum\limits_{k=-\infty}^{\infty}2^{k\alpha q} \|f\chi_k\|^q_{L^{p_2}_{\rm rad}L^{p_1}_{\rm ang}(\mathbb R^{n})}\Big\}^{1/q}.$$
\end{definition}

\begin{remark}\label{re2.8}
(1) ${\dot{K}}^{\alpha,q}_{L^{p}_{\rm rad}L^{p}_{\rm ang}}(\mathbb R^{n})=\dot{K}^{\alpha,q}_{p}(\mathbb R^{n})$, ${\dot{K}}^{0,p_2}_{L^{p_2}_{\rm rad}L^{p_1}_{\rm ang}}(\mathbb R^{n})=L^{p_2}_{\rm rad}L^{p_1}_{\rm ang}(\mathbb R^{n})$, ${\dot{K}}^{0,p}_{L^{p}_{\rm rad}L^{p}_{\rm ang}}(\mathbb R^{n})=L^p(\mathbb R^{n})$ and ${\dot{K}}^{\alpha/q,q}_{L^{q}_{\rm rad}L^{q}_{\rm ang}}(\mathbb R^{n})=L^q(|x|^{\alpha}dx)$.\\
(2) if $0<q_1\leq q_2<\infty$, then ${\dot{K}}^{\alpha,q_2}_{L^{p_2}_{\rm rad}L^{p_1}_{\rm ang}}(\mathbb R^{n})\subset{\dot{K}}^{\alpha,q_1}_{L^{p_2}_{\rm rad}L^{p_1}_{\rm ang}}(\mathbb R^{n})$.\\
(3) if $\alpha_2\leq \alpha_2$, then ${\dot{K}}^{\alpha_1,q}_{L^{p_2}_{\rm rad}L^{p_1}_{\rm ang}}(\mathbb R^{n})\subset{\dot{K}}^{\alpha_2,q}_{L^{p_2}_{\rm rad}L^{p_1}_{\rm ang}}(\mathbb R^{n})$.\\
(4) if $1<p_1\leq \tilde{p_1}<\infty$, then ${\dot{K}}^{\alpha_1,q}_{L^{{p_2}}_{\rm rad}L^{\tilde{p_1}}_{\rm ang}}(\mathbb R^{n})\subset{\dot{K}}^{\alpha_2,q}_{L^{p_2}_{\rm rad}L^{p_1}_{\rm ang}}(\mathbb R^{n})$.
\end{remark}

\begin{definition}\label{de2.9}  Let $\alpha\in\mathbb R$, $0<p_1,p_2, q<\infty$ and $\lambda\geq0$. The mixed radial-angular homogeneous Morrey-Herz spaces $M\dot{K}^{\alpha,q,\lambda}_{L^{p_2}_{\rm rad}L^{p_1}_{\rm ang}}(\mathbb R^{n})$
are defined by
$$M\dot{K}^{\alpha,q,\lambda}_{L^{p_2}_{\rm rad}L^{p_1}_{\rm ang}}(\mathbb R^{n})=\Big\{f\in L^{p_2}_{\rm rad,{loc}}L^{p_1}_{\rm ang}((0,\infty)\times{S}^{n-1}) \setminus{\{0\}}): \|f\|_{M\dot{K}^{\alpha,q,\lambda}_{L^{p_2}_{\rm rad}L^{p_1}_{\rm ang}}(\mathbb R^{n})}<\infty\Big\},$$
where
$$\|f\|_{{M\dot{K}}^{\alpha,q,\lambda}_{L^{p_2}_{\rm rad}L^{p_1}_{\rm ang}}(\mathbb R^{n})}=\sup\limits_{k_0\in\mathbb{Z}}2^{-k_0\lambda}\Big\{\sum\limits_{k=-\infty}^{k_0}2^{k\alpha q} \|f\chi_k\|^q_{L^{p_2}_{\rm rad}L^{p_1}_{\rm ang}(\mathbb R^{n})}\Big\}^{1/q}.$$
\end{definition}

\begin{remark}\label{re2.10} If $\lambda=0$, then ${M\dot{K}}^{\alpha,q,\lambda}_{L^{p_2}_{\rm rad}L^{p_1}_{\rm ang}}(\mathbb R^{n})={\dot{K}}^{\alpha,q}_{L^{p_2}_{\rm rad}L^{p_1}_{\rm ang}}(\mathbb R^{n})$.
\end{remark}

The main purpose in this paper is to consider some sharp bounds for Hardy type operators on these mixed radial-angular type function spaces. In addition, we  also will  establish the sharp weak-type estimates for the fractional Hardy operator and its conjugate operator, respectively.

In what follows, $A\simeq B$ denotes $A$ is equivalent to $B$, that means there exist two positive constants $C_1$ and $C_2$ such that $C_1A\leq B\leq C_2A$. The usual beta function $B(z,\omega)=\int_{0}^{1}t^{z-1}(1-t)^{\omega-1}dt$, where $z$ and $\omega$ are complex numbers with positive real parts, and the gamma function $\Gamma(z)=\int_{0}^{\infty}t^{z-1}e^{-t}dt$, where $z$ is a  complex number with positive real part, and these two functions have the following relationship:
$B(z,\omega)\Gamma(z+\omega)=\Gamma(z)\Gamma(\omega)$.

The rest of this paper is organized as follows. In Section 3, we will establish some sharp bounds for Hardy type operators on mixed radial-angular type function spaces.
Several kinds of sharp weak-type estimates for the fractional Hardy operator will be obtained in Section 4. Finally, we will give some sharp bounds of weighted Hardy-Littlewood averages  on mixed radial-angular Herz type function spaces. We would like to remark that the main ideas of our proofs
are taken from \cite{FL,FL1,FLLZ,GHC,YL}.

\section{Some sharp bounds of Hardy type operators on mixed radial-angular type function spaces}

In this section, we will establish some sharp bounds of Hardy type operators on mixed radial-angular type function spaces. Our main conclusions stated as  follows.
\begin{theorem}\label{th3.1}
Let $m\in\mathbb N$, $r\in (0,\infty)$, $1<p_i,\tilde{p}_{i}<\infty$, $1\leq p,\tilde{p}<\infty$,  $f\in L^{p_i}_{rad}L^{\tilde{p_i}}_{ang}(\mathbb R^{n}, r^{\frac{{\alpha}_ip_i}{p}})$, $i=1,\ldots,m$, $\frac{1}{p}=\frac{1}{p_1}+\cdots\frac{1}{p_m}$, $\alpha_i<(1-1/p_i)pn$ and $\alpha=\alpha_1+\cdots+\alpha_m$. Then $m$-linear Hardy operator $\mathcal{H}^{m}$ maps the product of weighted Lebesgue spaces  $L^{p_1}_{rad}L^{\tilde{p_1}}_{ang}(\mathbb R^{n}, r^{\frac{{\alpha}_1p_1}{p}})\times\cdots\times L^{p_m}_{rad}L^{\tilde{p_m}}_{ang}(\mathbb R^{n}, r^{\frac{{\alpha}_mp_m}{p}})$ to $L^{p}_{rad}L^{\tilde{p}}_{ang}(\mathbb R^{n}, r^{{\alpha}})$ with norm equal to the constant
$${{\omega}_n}^{\frac{1}{\tilde{p}}-\sum\limits_{i=1}^{m}\frac{1}{{\tilde{p}}_{i}}}\frac{\omega_n^m}{\omega_{mn}}\frac{pmn}{{pmn-n-\alpha}}\frac{1}{2^m-1}\frac{\prod_{i=1}^{m}\Gamma(\frac{n}{2}(1-\frac{1}{p_i}-\frac{\alpha_i}{p_i}))}{\Gamma(\frac{n}{2}(1-\frac{1}{p}-\frac{\alpha}{p}))}.$$
\end{theorem}

\begin{theorem}\label{th3.2}
Let $m\in\mathbb N$,  $f_i\in {\mathcal{B}L^{p_i,\lambda_i}_{\rm rad}L^{\tilde{p}}_{\rm ang}(\mathbb R^n)}$, $1<p_i, \tilde{p_i}<\infty$, $1<p<\infty$, $\frac{1}{p}=\frac{1}{p_1}+\cdots+\frac{1}{p_m}$, $-1/p_i\leq\lambda_i<0$, $i=1,\ldots,m$, and $\lambda=\lambda_1+\cdots+\lambda_m$. Then $m$-linear Hardy operator $\mathcal{H}^{m}$ maps ${\mathcal{B}L^{p_1,\lambda_1}_{\rm rad}L^{\tilde{p}_1}_{\rm ang}(\mathbb R^n)}\times\cdots\times {\mathcal{B}L^{p_m,\lambda_m}_{\rm rad}L^{\tilde{p}_m}_{\rm ang}(\mathbb R^n)}$ to  ${\mathcal{B}L^{p,\lambda}_{\rm rad}L^{\tilde{p}}_{\rm ang}(\mathbb R^n)}$ with norm equal to the constant
$${{\omega}_n}^{\frac{1}{\tilde{p}}-\sum\limits_{i=1}^{m}\frac{1}{{\tilde{p}}_{i}}}\frac{\omega_n^m}{\omega_{mn}}\frac{m}{\lambda+m}\frac{1}{2^m-1}\frac{\prod_{i=1}^{m}\Gamma(\frac{n}{2}(1+\lambda_i))}{\Gamma(\frac{n}{2}(m+\lambda))}.$$
\end{theorem}

\begin{theorem}\label{th3.3}  Let $1<p, \tilde{p}<\infty$ and $-1/p<\lambda<{1/n}$. Then
$$\|\mathcal{H}f\|_{{CMOL^{p,\lambda}_{\rm rad}L^{\tilde{p}}_{\rm ang}(\mathbb R^n)}}\leq\frac{1}{1+\lambda}\|f\|_{{CMOL^{p,\lambda}_{\rm rad}L^{\tilde{p}}_{\rm ang}(\mathbb R^n)}}.$$
Moreover,
$$\|\mathcal{H}\|_{{CMOL^{p,\lambda}_{\rm rad}L^{\tilde{p}}_{\rm ang}(\mathbb R^n)}\rightarrow {CMOL^{p,\lambda}_{\rm rad}L^{\tilde{p}}_{\rm ang}(\mathbb R^n)}}=\frac{1}{1+\lambda}.$$
Furthermore,
$$\|\mathcal{H}\|_{{CMOL^{p}_{\rm rad}L^{\tilde{p}}_{\rm ang}(\mathbb R^n)}\rightarrow {CMOL^{p}_{\rm rad}L^{\tilde{p}}_{\rm ang}(\mathbb R^n)}}=1.$$
\end{theorem}

\begin{proof}[Proof of Theorem \ref{th3.1}.]  In fact, we only need to provide the proof with the case $1\leq m\leq2$, and the similar procedure leads to the general case $m>2$. We follow some strategies in \cite{FLLZ}.

{\bf Case 1: Weighted case when $m=1$}

Set
$$g_{f}(x)=\frac{1}{{\omega}_{n}}\int_{{S}^{n-1}_{\xi}}f(|x|{\xi})d\sigma({\xi}), \quad x\in\mathbb R^{n},$$
where ${\omega}_{n}=2\pi^{n/2}/\Gamma(n/2)$.

It is not hard to see that the norm of the operator $\mathcal{H}_1$ is equal to the norm of the operator $\mathcal{H}_1$ restricts to the set of nonnegative radial functions. Therefore, it suffices to show that the case of the function $f$ is a nonnegative radial function.
\begin{align*}
&\|\mathcal{H}^1(f)\|_{L^{p}_{rad}L^{\tilde{p}_1}_{ang}(\mathbb R^{n}, r^{{\alpha}})}\\
&=\Big(\int_{0}^{\infty}\Big(\int_{{S}^{n-1}}\big|\mathcal{H}^1(f)(r\theta)\big|^{\tilde{p}}d\theta\Big)^{{p}/{\tilde{p}}}r^{n-1+\alpha}dr\Big)^{1/p}\\
&=\Big(\int_{0}^{\infty}\Big(\int_{{S}^{n-1}}\big|\frac{1}{\nu_nr^n}\int_{B(0,r)}f(y)dy\big|^{\tilde{p}}d\theta\Big)^{{p}/{\tilde{p}}}r^{n-1+\alpha}dr\Big)^{1/p}\\
&=\frac{{\omega_n}^{1/\tilde{p}}}{\nu_n}\Big(\int_{0}^{\infty}\big|\int_{B(0,1)}f(r|y|)dy\big|^{p}r^{n-1+\alpha}dr\Big)^{1/p}dy\\
&\leq\frac{{\omega_n}^{1/\tilde{p}}}{\nu_n}\int_{B(0,1)}\Big(\int_{0}^{\infty}|f(r|y|)|^{p}r^{n-1+\alpha}dr\Big)^{1/p}dy\\
&=\frac{{\omega_n}^{1/\tilde{p}-1/\tilde{p}_1}}{\nu_n}\|f\|_{L^{p}_{rad}L^{\tilde{p}_1}_{ang}(\mathbb R^{n}, r^{{\alpha}})}\int_{B(0,1)}|y|^{\frac{-n-\alpha}{p}}dy\\
&={\omega_n}^{1/\tilde{p}-1/\tilde{p}_1}\frac{\omega_n}{\nu_n}\frac{p}{pn-n-\alpha}\|f\|_{L^{p}_{rad}L^{\tilde{p}_1}_{ang}(\mathbb R^{n}, r^{{\alpha}})}\\
&={\omega_n}^{1/\tilde{p}-1/\tilde{p}_1}\frac{pn}{pn-n-\alpha}\|f\|_{L^{p}_{rad}L^{\tilde{p}_1}_{ang}(\mathbb R^{n}, r^{{\alpha}})}.
\end{align*}
This implies that
 $$\|\mathcal{H}^1\|_{L^{p}_{rad}L^{\tilde{p}_1}_{ang}(\mathbb R^{n}, r^{{\alpha}})\rightarrow L^{p}_{rad}L^{\tilde{p}}_{ang}(\mathbb R^{n}, r^{{\alpha}})}\leq{\omega_n}^{1/\tilde{p}-1/\tilde{p}_1}\frac{pn}{pn-n-\alpha}.$$

Conversely, for $0<\varepsilon<\min\{1,\frac{n}{p'}-\frac{\alpha}{p}\}$, we take $$f_{\varepsilon}(x)=|x|^{-\frac{n+\alpha}{p}-\varepsilon}\chi_{\{|x|>1\}}(x).$$
Then, we have
$$\|f_{\varepsilon}\|_{L^{p}_{rad}L^{\tilde{p}_1}_{ang}(\mathbb R^{n}, r^{\alpha})}=\frac{{{\omega}_{n}}^{1/\tilde{p}_1}}{({p\varepsilon})^{1/p}}.$$
It follows that
 \begin{align*}
 \mathcal{H}^1(f_\varepsilon)(x)=\frac{1}{\nu_n|x|^{-\frac{n+\alpha}{p}-\varepsilon}}\int_{1/|x|<|y|<1}|y|^{-\frac{n+\alpha}{p}-\varepsilon}dy\chi_{\{|x|>1\}}(x).
 \end{align*}
Therefore, a simple calculation deduces that

\begin{align*}
&\|\mathcal{H}^1(f_\varepsilon)\|_{L^{p}_{rad}L^{\tilde{p}}_{ang}(\mathbb R^{n}, r^{{\alpha}})}\\
&=\Big(\int_{1}^{\infty}\Big(\int_{{S}^{n-1}}\big|\mathcal{H}^1(f_\varepsilon)(r\theta)\big|^{\tilde{p}}d\theta\Big)^{{p}/{\tilde{p}}}r^{n-1+\alpha}dr\Big)^{1/p}\\
&=\Big(\int_{1}^{\infty}\Big(\int_{{S}^{n-1}}\big|\frac{r^{-\frac{n+\alpha}{p}-\varepsilon}}{\nu_n}\int_{1/r<|y|<1}|y|^{-\frac{n+\alpha}{p}-\varepsilon}dy\big|^{\tilde{p}}d\theta\Big)^{{p}/{\tilde{p}}}r^{n-1+\alpha}dr\Big)^{1/p}\\
&\geq\frac{{\omega_n}^{1/\tilde{p}}}{\nu_n}\Big(\int_{1/\varepsilon}^{\infty}\Big(\int_{{S}^{n-1}}\big|\frac{r^{-\frac{n+\alpha}{p}-\varepsilon}}{\nu_n}\int_{\varepsilon<|y|<1}|y|^{-\frac{n+\alpha}{p}-\varepsilon}dy\big|^{\tilde{p}}d\theta\Big)^{{p}/{\tilde{p}}}r^{n-1+\alpha}dr\Big)^{1/p}\\
&=\frac{{\omega_n}^{1/\tilde{p}+1}}{\nu_n}\Big(\int_{1/\varepsilon}^{\infty}\Big(\int_{\varepsilon<|y|<1}|y|^{-\frac{n+\alpha}{p}-\varepsilon}dy\Big)^{p}r^{-\varepsilon p-1}dr\Big)^{1/p}\\
&=\frac{{\omega_n}^{1/\tilde{p}+1}}{\nu_n}\frac{1-\varepsilon^{n-\frac{n}{p}-\frac{\alpha}{p}-\varepsilon}}{n-\frac{n}{p}-\frac{\alpha}{p}-\varepsilon}\frac{\varepsilon^{\varepsilon}}{({p\varepsilon})^{1/p}}\\
&={\omega_n}^{1/\tilde{p}-1/\tilde{p}_1}\frac{{\omega_n}}{\nu_n}\frac{1-\varepsilon^{n-\frac{n}{p}-\frac{\alpha}{p}-\varepsilon}}{n-\frac{n}{p}-\frac{\alpha}{p}-\varepsilon}{\varepsilon^{\varepsilon}}\|f_{\varepsilon}\|_{L^{p}_{rad}L^{\tilde{p}}_{ang}(\mathbb R^{n}, r^{\alpha})}.
\end{align*}
 Consequently,
 $$\|\mathcal{H}^1\|_{L^{p}_{rad}L^{\tilde{p}}_{ang}(\mathbb R^{n}, r^{{\alpha}})\rightarrow L^{p}_{rad}L^{\tilde{p}}_{ang}(\mathbb R^{n}, r^{{\alpha}})}\geq{\omega_n}^{1/\tilde{p}-1/\tilde{p}_1}\frac{{\omega_n}}{\nu_n}\frac{1-\varepsilon^{n-\frac{n}{p}-\frac{\alpha}{p}-\varepsilon}}{n-\frac{n}{p}-\frac{\alpha}{p}-\varepsilon}{\varepsilon^{\varepsilon}},$$
 and letting  $\varepsilon\rightarrow0$, we deduce
 $$\|\mathcal{H}^1\|_{L^{p}_{rad}L^{\tilde{p}}_{ang}(\mathbb R^{n}, r^{{\alpha}})\rightarrow L^{p}_{rad}L^{\tilde{p}}_{ang}(\mathbb R^{n}, r^{{\alpha}})}\geq{\omega_n}^{1/\tilde{p}-1/\tilde{p}_1}\frac{pn}{pn-n-\alpha}.$$

{\bf Case 2: Unweighted case when $m=2$.}

$$g_{f}^i(x_i)=\frac{1}{{\omega}_{n}}\int_{{S}^{n-1}_{\xi_i}}f(|x_i|{\xi}_i)d\sigma({\xi}_i), \quad x_i\in\mathbb R^{n}, i=1,2,$$
where as before ${\omega}_{n}=2\pi^{n/2}/\Gamma(n/2)$. It follows from that $\mathcal{H}^2(g_{f}^1,g_{f}^2)(x)$ is equal to
\begin{align*}
\mathcal{H}^2(g_{f}^1,g_{f}^2)(x)=\mathcal{H}^2(f_1,f_2)(x).
\end{align*}
The Minkowski inequality  gives rise to the following

\begin{align*}
\|g_{f}^i\|_{L^{p}_{rad}L^{\tilde{p}}_{ang}(\mathbb R^{n})}\leq\|f^i\|_{L^{p}_{rad}L^{\tilde{p}}_{ang}(\mathbb R^{n})}.
\end{align*}
Therefore,
\begin{align*}
\frac{\|\mathcal{H}^2(f_1,f_2)\|_{L^{p}_{rad}L^{\tilde{p}}_{ang}(\mathbb R^{n})}}{\|f_1\|_{L^{p_1}_{rad}L^{\tilde{p_1}}_{ang}(\mathbb R^{n})}\|f_2\|_{L^{p_2}_{rad}L^{\tilde{p_2}}_{ang}(\mathbb R^{n})}}\leq\frac{\|\mathcal{H}^2(g_{f}^1,g_{f}^2)\|_{L^{p}_{rad}L^{\tilde{p}}_{ang}(\mathbb R^{n} )}}{\|g_{f}^1\|_{L^{p_1}_{rad}L^{\tilde{p_1}}_{ang}(\mathbb R^{n})}\|g_{f}^2\|_{L^{p_2}_{rad}L^{\tilde{p_2}}_{ang}(\mathbb R^{n})}}.
\end{align*}
This implies that the norm of the operator of $\mathcal{H}^2$ is equal to the norm of the operator $\mathcal{H}^2$ restricts to the set of nonnegative radial functions. Therefore, it suffices to show that the case of the function $f_i(i=1,2)$ is a nonnegative radial function.

By Minkowski's  inequality  and H\"{o}lder's inequality, we have
\begin{align*}
&\|\mathcal{H}^2(f_1,f_2)\|_{L^{p}_{rad}L^{\tilde{p}}_{ang}(\mathbb R^{n})}\\
&=\Bigg(\int_{0}^{\infty}\Big(\int_{{S}^{n-1}_{x}}|\mathcal{H}^2(f_1,f_2)(r\theta)|^{\tilde{p}}d\sigma({\theta})\Big)^{p/\tilde{p}} r^{n-1}dr\Bigg)^{1/p}\\
&=\Bigg(\int_{0}^{\infty}\Big(\int_{{S}^{n-1}_{x}}|\frac{1}{\nu_{2n}r^{2n}}\int_{|(y_1,y_2)|<r}f_1(y_1)f_2(y_2)dy_1dy_2|^{\tilde{p}}d\sigma({\theta})\Big)^{p/{\tilde{p}}} r^{n-1}dr\Bigg)^{1/p}\\
&=\Bigg(\int_{0}^{\infty}\Big(\int_{{S}^{n-1}_{x}}|\frac{1}{\nu_{2n}}\int_{|(z_1,z_2)|<1}f_1(r|z_1|)f_2(r|z_2|)dz_1dz_2|^{\tilde{p}}d\sigma({\theta})\Big)^{p/{\tilde{p}}} r^{n-1}dr\Bigg)^{1/p}\\
&=\frac{{\omega_n}^{1/{\tilde{p}}}}{\nu_{2n}}\Big(\int_{0}^{\infty}|\int_{|(z_1,z_2)|<1}f_1(rz_1)f_2(r|z_2|)dz_1dz_2|^{p} r^{n-1}dr\Big)^{1/p}\\
&\leq\frac{{\omega_n}^{1/{\tilde{p}}}}{\nu_{2n}}\int_{|(z_1,z_2)|<1}\Big(\int_{0}^{\infty}|f_1(r|z_1|)f_2(r|z_2|)|^p r^{n-1}dr\Big)^{1/p}dz_1dz_2\\
&=\frac{{\omega_n}^{1/{\tilde{p}}}}{\nu_{2n}}\int_{|(z_1,z_2)|<1}\prod\limits_{i=1}^{2}\Big(\int_{0}^{\infty}|f_i(r)|^{p_i} r^{n-1}dr\Big)^{1/p_i}|z_1|^{-\frac{n}{p_1}}|z_2|^{-\frac{n}{p_2}}dz_1dz_2\\
&={\omega_n}^{\frac{1}{\tilde{p}}-\sum\limits_{i=1}^{2}\frac{1}{{\tilde{p}}_{i}}}\frac{1}{\nu_{2n}}\int_{|(z_1,z_2)|<1}|z_1|^{-\frac{n}{p_1}}|z_2|^{-\frac{n}{p_2}}dz_1dz_2\|f_1\|_{L^{p_1}_{rad}L^{\tilde{p_1}}_{ang}(\mathbb R^{n})}\|f_2\|_{L^{p_2}_{rad}L^{\tilde{p_2}}_{ang}(\mathbb R^{n})}\\
&={\omega_n}^{\frac{1}{\tilde{p}}-\sum\limits_{i=1}^{2}\frac{1}{{\tilde{p}}_{i}}}\frac{1}{\nu_{2n}}\frac{\omega_n^2}{2n}\frac{p}{2p-1}B(\frac{n}{2}-\frac{n}{2p_1},\frac{n}{2}-\frac{n}{2p_2})\|f_1\|_{L^{p_1}_{rad}L^{\tilde{p_1}}_{ang}(\mathbb R^{n})}\|f_2\|_{L^{p_2}_{rad}L^{\tilde{p_2}}_{ang}(\mathbb R^{n})}\\
&={\omega_n}^{\frac{1}{\tilde{p}}-\sum\limits_{i=1}^{2}\frac{1}{{\tilde{p}}_{i}}}\frac{\omega_n^2}{\omega_{2n}}\frac{p}{2p-1}\frac{\prod\limits_{i=1}^{2}\Gamma(\frac{n}{2}(1-\frac{1}{p_i}))}{\Gamma(\frac{n}{2}(2-\frac{1}{p}))}\|f_1\|_{L^{p}_{rad}L^{\tilde{p_1}}_{ang}(\mathbb R^{n})}\|f_2\|_{L^{p}_{rad}L^{\tilde{p_2}}_{ang}(\mathbb R^{n})}.
\end{align*}
This means that
 $$\|\mathcal{H}^2\|_{L^{p_1}_{rad}L^{\tilde{p_1}}_{ang}(\mathbb R^{n})\times L^{p_2}_{rad}L^{\tilde{p_2}}_{ang}(\mathbb R^{n})\rightarrow L^{p}_{rad}L^{\tilde{p}}_{ang}(\mathbb R^{n})}\leq{\omega_n}^{\frac{1}{\tilde{p}}-\sum\limits_{i=1}^{2}\frac{1}{{\tilde{p}}_{i}}}\frac{\omega_n^2}{\omega_{2n}}\frac{p}{2p-1}\frac{\prod\limits_{i=1}^{2}\Gamma(\frac{n}{2}(1-\frac{1}{p_i}))}{\Gamma(\frac{n}{2}(2-\frac{1}{p}))}.$$

Now, for $0<\varepsilon<\min\{1,\frac{(p_1-1)n}{p_2},\frac{n}{p_2'}\}$, we define
$$f_{\varepsilon}^1(x_1)=|x_1|^{-\frac{n}{p_1}-\frac{p_2\varepsilon}{p_1}}\chi_{\{|x_1|>\frac{\sqrt{2}}{2}\}}(x_1),\quad f_{\varepsilon}^2(x_2)=|x_2|^{-\frac{n}{p_2}-\varepsilon}\chi_{\{|x_2|>\frac{\sqrt{2}}{2}\}}(x_2).$$
By an elementary calculation, we obtain that
$$\|f_{\varepsilon}^1\|_{L^{p_1}_{rad}L^{\tilde{p_1}}_{ang}(\mathbb R^{n})}=\frac{{{\omega}_{n}}^{1/\tilde{p_1}}}{({p_2\varepsilon})^{1/p_1}}(\sqrt{2})^{\frac{p_2\varepsilon}{p_1}},\quad \|f_{\varepsilon}\|_{L^{p_2}_{rad}L^{\tilde{p_2}}_{ang}(\mathbb R^{n})}=\frac{{{\omega}_{n}}^{1/\tilde{p_2}}}{({p_2\varepsilon})^{1/p_2}}(\sqrt{2})^{\varepsilon}.$$
Therefore,
\begin{align*}
&\mathcal{H}^2(f_{\varepsilon}^1,f_{\varepsilon}^2)(x)\\
&=\frac{1}{\nu_{2n}}|x|^{-\frac{n+p_2\varepsilon}{p}}\int_{\{|(y_1,y_2)|<1;|y_1|>\frac{\sqrt{2}}{2|x|},|y_2|>\frac{\sqrt{2}}{2|x|}\}}|y_1|^{-\frac{n+p_2\varepsilon}{p_1}}|y_2|^{-\frac{n}{p_2}-\varepsilon}dy_1dy_2\chi_{\{|x|>1\}}(x).
\end{align*}
This gives that
\begin{align*}
&\|\mathcal{H}^2(f_{\varepsilon}^1,f_{\varepsilon}^2)\|_{L^{p}_{rad}L^{\tilde{p}}_{ang}(\mathbb R^{n})}\\
&={{\omega}_{n}}^{1/\tilde{p}}\Bigg(\int_{1}^{\infty}\Big(\int_{\{|(y_1,y_2)|<1;|y_1|>\frac{\sqrt{2}}{2|x|},|y_2|>\frac{\sqrt{2}}{2|x|}\}}|y_1|^{-\frac{n+p_2\varepsilon}{p_1}}|y_2|^{-\frac{n}{p_2}-\varepsilon}dy_1dy_2\Big)^{p} r^{-p_2\varepsilon-1}dr\Bigg)^{1/p}\\
&\geq{{\omega}_{n}}^{1/\tilde{p}}\Bigg(\int_{1/\varepsilon}^{\infty}\Big(\int_{\{|(y_1,y_2)|<1;|y_1|>\frac{\sqrt{2}}{2\varepsilon},|y_2|>\frac{\sqrt{2}}{2\varepsilon}\}}|y_1|^{-\frac{n+p_2\varepsilon}{p_1}}|y_2|^{-\frac{n}{p_2}-\varepsilon}dy_1dy_2\Big)^{p} r^{-p_2\varepsilon-1}dr\Bigg)^{1/p}\\
&={\omega_n}^{\frac{1}{\tilde{p}}-\sum\limits_{i=1}^{2}\frac{1}{{\tilde{p}}_{i}}}\frac{\omega_n^2}{n-\frac{n}{p_2}-\varepsilon}\Big[\frac{1}{2}B(\frac{n}{2}-\frac{n}{2p_1}-\frac{p_2\varepsilon}{2p_1},\frac{n-\varepsilon}{2}-\frac{n}{2p_2}+1)\\
&\quad-\frac{1}{2}\int^{\frac{\varepsilon^2}{2}}_{0}(1-t)^{\frac{1}{2}(n-\frac{n}{p_2}-\varepsilon)}t^{\frac{1}{2}(n-\frac{n}{p_1}-\frac{p_2\varepsilon}{p_1})-1}dt\\
&\qquad-\big(\frac{\sqrt{2}\varepsilon}{2}\big)^{n-\frac{n}{p_2}}\frac{(\sqrt{2})^\varepsilon}{{\varepsilon^{\varepsilon}}}\frac{1}{n-\frac{n+p_2\varepsilon}{p_1}}\Big(1-\big(\frac{\sqrt{2\varepsilon}}{2}\big)^{n-\frac{n}{p_1}}\big(\frac{\sqrt{2}}{2}\big)^{-\frac{p_2\varepsilon}{p_1}}({\varepsilon^{\varepsilon}})^{-\frac{p_2}{p_1}}\Big)\Big]\\
&\qquad\quad\times({\varepsilon}^{\varepsilon})^{\frac{p_2}{p}}\big(\frac{\sqrt{2}}{2}\big)^{\frac{p_2\varepsilon}{p}}\|f_1^\varepsilon\|_{L^{p}_{rad}L^{\tilde{p_1}}_{ang}(\mathbb R^{n})}\|f_2^\varepsilon\|_{L^{p}_{rad}L^{\tilde{p_2}}_{ang}(\mathbb R^{n})}.
\end{align*}
Let $\varepsilon\rightarrow0$, we have
$$\|\mathcal{H}^2\|_{L^{p}_{rad}L^{\tilde{p}}_{ang}(\mathbb R^{n})}\geq{\omega_n}^{\frac{1}{\tilde{p}}-\sum\limits_{i=1}^{2}\frac{1}{{\tilde{p}}_{i}}}\frac{\omega_n^2}{\omega_{2n}}\frac{p}{2p-1}\frac{\prod\limits_{i=1}^{2}\Gamma(\frac{n-\varepsilon}{2}(1-\frac{1}{p_i}))}{\Gamma(\frac{n}{2}(2-\frac{1}{p}))}.$$

For weighted case, the upper bounded in this case is similar to that of the previous case, for the lower bound, we only take
$$f_{\varepsilon}^1(x_1)=|x_1|^{-\frac{n}{p_1}-\frac{p_2\varepsilon}{p_1}-\frac{\alpha}{p_1}}\chi_{\{|x_1|>\frac{\sqrt{2}}{2}\}}(x_1),\quad f_{\varepsilon}^2(x_2)=|x_2|^{-\frac{n}{p_2}-\frac{\alpha}{p_2}-\varepsilon}\chi_{\{|x_2|>\frac{\sqrt{2}}{2}\}}(x_2).$$

This proves Theorem \ref{th3.1}.

\end{proof}

Now we turn to prove Theorem \ref{th3.2} in this  position.

\begin{proof}[Proof of Theorem \ref{th3.2}.]
We only consider the cases $m=1$ and $m=2$, the rest of cases does not involve any new ideas, thus we omit the details.

{\bf Case 1: $m=1$.} According to the proof of Theorem \ref{th3.1}, we know that the operator $\mathcal{H}^1$ and its restriction to radial functions have the same operator norm  on the spaces ${\mathcal{B}L^{p,\lambda}_{\rm rad}L^{\tilde{p}}_{\rm ang}(\mathbb R^n)}$.

Let
$$\mathcal{H}^1(f)(x)=\frac{1}{{\nu}_{n}}\int_{B(0,1)}f(t|x|)dt.$$
By Minkowski's  inequality  and H\"{o}lder's inequality, we have

\begin{align*}
&\|\mathcal{H}^1(f)\|_{{\mathcal{B}L^{p,\lambda}_{\rm rad}L^{\tilde{p}}_{\rm ang}(\mathbb R^n)}}\\
&=\sup\limits_{r>0}\Big(\frac{1}{\nu_nr^{n+n\lambda p}}\int_{0}^{r}\Big(\int_{{S}^{n-1}}|\frac{1}{{\nu}_{n}}\int_{B(0,1)}f(t|\rho\theta|)dt|^{\tilde{p}}d\sigma(\theta)\Big)^{p/\tilde{p}}\rho^{n-1}d\rho\Big)^{1/p}\\
&\leq{\omega_n}^{\frac{1}{\tilde{p}}}\sup\limits_{r>0}\Big(\frac{1}{{\nu_n}^{1+p}r^{n+n\lambda p}}\int_{0}^{r}\Big(\int_{B(0,1)}|f(t\rho)|dt\Big)^{p}\rho^{n-1}d\rho\Big)^{1/p}\\
&\leq\frac{{\omega_n}^{\frac{1}{\tilde{p}}}}{\nu_n}\int_{B(0,1)}\sup\limits_{r>0}\Big(\frac{1}{{\nu_n}r^{n+n\lambda p}}\int_{0}^{r}|f(|t|\rho)|^{p}\rho^{n-1}d\rho\Big)^{1/p}dt\\
&\leq\frac{{\omega_n}^{\frac{1}{\tilde{p}}-\frac{1}{\tilde{p}_1}}}{\nu_n}\int_{B(0,1)}|t|^{n\lambda}dt\|f\|_{{{\mathcal{B}L^{p,\lambda}_{\rm rad}L^{\tilde{p}_1}_{\rm ang}(\mathbb R^n)}}}\\
&={\omega_n}^{\frac{1}{\tilde{p}}-\frac{1}{\tilde{p}_1}}\frac{1}{1+\lambda}\|f\|_{{{\mathcal{B}L^{p,\lambda}_{\rm rad}L^{\tilde{p}_1}_{\rm ang}(\mathbb R^n)}}}.
\end{align*}
On the other hand, we take $\tilde{f}(x)=|x|^{n\lambda}$, then  $\tilde{f}\in{{\mathcal{B}L^{p,\lambda}_{\rm rad}L^{\tilde{p}_1}_{\rm ang}(\mathbb R^n)}}$, and
$$\|\mathcal{H}^1(f)\|_{{\mathcal{B}L^{p,\lambda}_{\rm rad}L^{\tilde{p}}_{\rm ang}(\mathbb R^n)}}=\frac{1}{1+\lambda}\|\tilde{f}\|_{{\mathcal{B}L^{p,\lambda}_{\rm rad}L^{\tilde{p}}_{\rm ang}(\mathbb R^n)}}={\omega_n}^{\frac{1}{\tilde{p}}-\frac{1}{\tilde{p}_1}}\frac{1}{1+\lambda}\|\tilde{f}\|_{{\mathcal{B}L^{p,\lambda}_{\rm rad}L^{\tilde{p}_1}_{\rm ang}(\mathbb R^n)}}.$$
This finishes the case of  $m=1$.

{\bf Case 1: $m=2$.}  We also note that the operator $\mathcal{H}^2$ and its restriction to radial functions have the same operator norm  on the spaces ${\mathcal{B}L^{p,\lambda}_{\rm rad}L^{\tilde{p}}_{\rm ang}(\mathbb R^n)}$. For the radial functions $f_1$ and $f_2$, we write
$$\mathcal{H}^2(f_1,f_2)(x)=\frac{1}{\nu_{2n}}\int_{|(y_1,y_2)|<1}f_1(|x|y_1)f_2(|x|y_2)dy_1dy_2.$$
The Minkowski  inequality  and H\"{o}lder inequality give rise to
\begin{align*}
&\|\mathcal{H}^1(f)\|_{{\mathcal{B}L^{p,\lambda}_{\rm rad}L^{\tilde{p}}_{\rm ang}(\mathbb R^n)}}\\
&=\frac{{\omega_n}^{\frac{1}{\tilde{p}}}}{\nu_{2n}}\sup\limits_{r>0}\Big(\frac{1}{{\nu_n}r^{n+n\lambda p}}\int_{0}^{r}\Big(\int_{|(y_1,y_2)|<1}f_1(\rho y_1)f_2(\rho y_2)dy_1dy_2\Big)^{p}\rho^{n-1}d\rho\Big)^{1/p}\\
&\leq\frac{{\omega_n}^{\frac{1}{\tilde{p}}}}{\nu_{2n}}\int_{|(y_1,y_2)|<1}\sup\limits_{r>0}\Big(\frac{1}{{\nu_n}r^{n+n\lambda p}}\int_{0}^{r}|f_1(\rho y_1)f_2(\rho y_2)|^{p}\rho^{n-1}d\rho\Big)^{1/p}dy_1dy_2\\
&\leq\frac{{\omega_n}^{\frac{1}{\tilde{p}}}}{\nu_{2n}}\int_{|(y_1,y_2)|<1}\sup\limits_{r>0}\Big(\frac{1}{{\nu_n}r^{n+n\lambda p_1}}\int_{0}^{r}|f_1(\rho y_1)|^{p_1}\rho^{n-1}d\rho\Big)^{1/p_1}\\
&\quad\times\sup\limits_{r>0}\Big(\frac{1}{{\nu_{2n}}r^{n+n\lambda p_2}}\int_{0}^{r}|f_2(\rho y_2)|^{p_2}\rho^{n-1}d\rho\Big)^{1/p_2}dy_1dy_2\\
&\leq\frac{{\omega_n}^{\frac{1}{\tilde{p}}}}{\nu_{2n}}\int_{|(y_1,y_2)|<1}\sup\limits_{r>0}\Big(\frac{1}{{\nu_n}(r|y_1|)^{n+n\lambda p_1}}\int_{0}^{r|y_1|}|f_1(\rho )|^{p_1}\rho^{n-1}d\rho\Big)^{1/p_1}\\
&\quad\times\sup\limits_{r>0}\Big(\frac{1}{{\nu_n}(r|y_2|)^{n+n\lambda p_2}}\int_{0}^{r|y_2|}|f_2(\rho)|^{p_2}\rho^{n-1}d\rho\Big)^{1/p_2}|y_1|^{n\lambda_1}|y_2|^{n\lambda_2}dy_1dy_2\\
\end{align*}

\begin{align*}
&\leq\frac{{\omega_n}^{\frac{1}{\tilde{p}}-\frac{1}{\tilde{p}_1}-\frac{1}{\tilde{p}_2}}}{\nu_{2n}}\int_{|(y_1,y_2)|<1}|y_1|^{n\lambda_1}|y_2|^{n\lambda_2}dy_1dy_2\|f_1\|_{{{\mathcal{B}L^{p,\lambda}_{\rm rad}L^{\tilde{p}_1}_{\rm ang}(\mathbb R^n)}}}\|f_2\|_{{{\mathcal{B}L^{p,\lambda}_{\rm rad}L^{\tilde{p}_2}_{\rm ang}(\mathbb R^n)}}}\\
&=\frac{{\omega_n}^{\frac{1}{\tilde{p}}-\frac{1}{\tilde{p}_1}-\frac{1}{\tilde{p}_2}}}{\nu_{2n}}\frac{\omega_n^2}{2n}\frac{1}{2+\lambda}B(\frac{n}{2}(1+\lambda_1),\frac{n}{2}(1+\lambda_2))
\|f_1\|_{{{\mathcal{B}L^{p,\lambda}_{\rm rad}L^{\tilde{p}_1}_{\rm ang}(\mathbb R^n)}}}\|f_2\|_{{{\mathcal{B}L^{p,\lambda}_{\rm rad}L^{\tilde{p}_2}_{\rm ang}(\mathbb R^n)}}}\\
&={\omega_n}^{\frac{1}{\tilde{p}}-\sum\limits_{i=1}^{2}\frac{1}{{\tilde{p}}_{i}}}\frac{\omega_n^2}{\omega_{2n}}\frac{1}{2+\lambda}\frac{\prod\limits_{i=1}^{2}\Gamma(\frac{n}{2}(1+\lambda_i))}{\Gamma(\frac{n}{2}(2+\lambda_i))}\|f_1\|_{{{\mathcal{B}L^{p,\lambda}_{\rm rad}L^{\tilde{p}_1}_{\rm ang}(\mathbb R^n)}}}\|f_2\|_{{{\mathcal{B}L^{p,\lambda}_{\rm rad}L^{\tilde{p}_2}_{\rm ang}(\mathbb R^n)}}}.
\end{align*}
On the other hand, we take $\tilde{f}_i(x)=|x|^{n\lambda_i}$, then
$$\|f_i\|_{{{\mathcal{B}L^{p,\lambda}_{\rm rad}L^{\tilde{p}_i}_{\rm ang}(\mathbb R^n)}}}=\frac{{\omega_n}^{\frac{1}{\tilde{p}_i}}}{n(1+\lambda_ip_i)}, \quad i=1,2.$$
It is not hard to check that
\begin{align*}
&\|\mathcal{H}^2(\tilde{f}_1,\tilde{f}_2)\|_{{\mathcal{B}L^{p,\lambda}_{\rm rad}L^{\tilde{p}}_{\rm ang}(\mathbb R^n)}}\\
&=\frac{1}{\nu_{2n}}\int_{|(y_1,y_2)|<1}|y_1|^{n\lambda_1}|y_2|^{n\lambda_2}dy_1dy_2\\
&={\omega_n}^{\frac{1}{\tilde{p}}-\sum\limits_{i=1}^{2}\frac{1}{{\tilde{p}}_{i}}}\frac{\omega_n^2}{\omega_{2n}}\frac{1}{2+\lambda}\frac{\prod\limits_{i=1}^{2}\Gamma(\frac{n}{2}(1+\lambda_i))}{\Gamma(\frac{n}{2}(2+\lambda_i))}\|f_1\|_{{{\mathcal{B}L^{p,\lambda}_{\rm rad}L^{\tilde{p}_1}_{\rm ang}(\mathbb R^n)}}}\|f_2\|_{{{\mathcal{B}L^{p,\lambda}_{\rm rad}L^{\tilde{p}_2}_{\rm ang}(\mathbb R^n)}}}.
\end{align*}
This completes the case of  $m=2$. Therefore, the proof of Theorem \ref{th3.2} is finished.
\end{proof}
Next, we come to prove Theorem \ref{th3.3}.

\begin{proof}[Proof of Theorem \ref{th3.3}.] Set
$$g_{f}(x)=\frac{1}{{\omega}_{n}}\int_{{S}^{n-1}_{\xi}}f(|x|{\xi})d\sigma({\xi}), \quad x\in\mathbb R^{n}.$$
Similarly, we have
\begin{align*}
\frac{\|\mathcal{H}(f)\|_{{CMOL^{p,\lambda}_{\rm rad}L^{\tilde{p}}_{\rm ang}(\mathbb R^n)}}}{\|f\|_{{CMOL^{p,\lambda}_{\rm rad}L^{\tilde{p}}_{\rm ang}(\mathbb R^n)}}}\leq\frac{\|\mathcal{H}(g_{f})\|_{{CMOL^{p,\lambda}_{\rm rad}L^{\tilde{p}}_{\rm ang}(\mathbb R^n)}}}{\|g_{f}\|_{{CMOL^{p,\lambda}_{\rm rad}L^{\tilde{p}}_{\rm ang}(\mathbb R^n)}}}.
\end{align*}
Note that
$$\mathcal{H}(f)_{B(0,R)}=\frac{1}{{\nu}_{n}}\int_{B(0,1)}f_{B(0,|t|R)}dt.$$
Then, we have
\begin{align*}
&\|\mathcal{H}(f)\|_{{CMOL^{p,\lambda}_{\rm rad}L^{\tilde{p}}_{\rm ang}(\mathbb R^n)}}\\
&=\sup\limits_{r>0}\Big(\frac{1}{\nu_nr^{n+n\lambda p}}\int_{0}^{r}\Big(\int_{{S}^{n-1}}|\mathcal{H}(f)(\rho\theta)-\mathcal{H}(f)_{B(0,r)}|^{\tilde{p}}d\sigma(\theta)\Big)^{p/\tilde{p}}\rho^{n-1}d\rho\Big)^{1/p}\\
&\leq\frac{1}{{\nu}_{n}}\sup\limits_{r>0}\Big(\frac{1}{\nu_nr^{n+n\lambda p}}\int_{0}^{r}\Big(\int_{{S}^{n-1}}|\int_{B(0,1)}\Big(f(|t|\rho\theta)-f_{B(0,|t|R)}\Big)dt|^{\tilde{p}}d\sigma(\theta)\Big)^{p/\tilde{p}}\rho^{n-1}d\rho\Big)^{1/p}\\
\end{align*}

\begin{align*}
&\leq\frac{1}{{\nu}_{n}}\sup\limits_{r>0}\Big(\frac{1}{\nu_nr^{n+n\lambda p}}\int_{0}^{r}\Big(\int_{B(0,1)}\Big(\int_{{S}^{n-1}}\big|f(|t|\rho\theta)-f_{B(0,|t|R)}\big|^{\tilde{p}}d\sigma(\theta)\Big)^{1/\tilde{p}}dt\Big)^{p}\rho^{n-1}d\rho\Big)^{1/p}\\
&\leq\frac{1}{{\nu}_{n}}\int_{B(0,1)}\sup\limits_{r>0}\Big(\frac{1}{\nu_nr^{n+n\lambda p}}\int_{0}^{r}\Big(\int_{{S}^{n-1}}\big|f(|t|\rho\theta)-f_{B(0,|t|R)}\big|^{\tilde{p}}d\sigma(\theta)\Big)^{p/\tilde{p}}\rho^{n-1}d\rho\Big)^{1/p}dt\\
&\leq\frac{1}{{\nu}_{n}}\int_{B(0,1)}\sup\limits_{r>0}\Big(\frac{1}{\nu_n(|t|r)^{n+n\lambda p}}\int_{0}^{|t|r}\Big(\int_{{S}^{n-1}}\big|f(\rho\theta)-f_{B(0,|t|R)}\big|^{\tilde{p}}d\sigma(\theta)\Big)^{p/\tilde{p}}\rho^{n-1}d\rho\Big)^{1/p}|t|^{n\lambda}dt\\
&\leq\frac{1}{{\nu}_{n}}\|f\|_{{CMOL^{p,\lambda}_{\rm rad}L^{\tilde{p}}_{\rm ang}(\mathbb R^n)}}\int_{B(0,1)}|t|^{n\lambda}dt\\
&=\frac{1}{1+\lambda}\|f\|_{{CMOL^{p,\lambda}_{\rm rad}L^{\tilde{p}}_{\rm ang}(\mathbb R^n)}}.
\end{align*}
On the other hand, we take
\begin{eqnarray*}
f_0(x)=
\begin{cases}
|x|^{n\lambda},&\mbox{}x\in\mathbb R^n_r,\\
(-1)^{n}|x|^{n\lambda},&\mbox{}x\in\mathbb R^n_l.
\end{cases}
\end{eqnarray*}
where $\mathbb R^n_r$ and  $\mathbb R^n_l$ represent  the left and right parts of $\mathbb R^n$ divided by the hyperplane $x_1$, by a simple calculation, we get
$$f_0\in{CMOL^{p,\lambda}_{\rm rad}L^{\tilde{p}}_{\rm ang}(\mathbb R^n)},\quad (f_0)_{B(0,r)}=0, \quad \mathcal{H}(f_0)=\frac{1}{1+\lambda}f_0.$$
Therefore, we have
$$\|\mathcal{H}f\|_{{CMOL^{p,\lambda}_{\rm rad}L^{\tilde{p}}_{\rm ang}(\mathbb R^n)}}\leq\frac{1}{1+\lambda}\|f\|_{{CMOL^{p,\lambda}_{\rm rad}L^{\tilde{p}}_{\rm ang}(\mathbb R^n)}}.$$
The proof of Theorem \ref{th3.3} is completed.
\end{proof}

\section{Sharp weak-type estimates for the fractional Hardy operator}

Our aim of this section is to consider the mixed radial-angular integrabilities for the fractional Hardy operator.
Recall that, for a nonnegative measurable function $f$ on $\mathbb R^{n}$, the $n$-dimensional fractional
Hardy operator $\mathrm{H}_{\beta}$ with spherical mean is defined by
$$\mathrm{H}_{\beta}(f)(x)= \frac{1}{|B(0,|x|)|^{1-\frac{\beta}{n}}}\int_{|y|\leq|x|}f(y)dy, \quad x\in\mathbb{R}^n\setminus \{0\}\;\text{and}\; 0<\beta<n.$$
Clearly, $$\mathrm{H}_{\beta}(f)(x)\leq C \mathrm{M}_{\beta}(f)(x),$$
where
$$\mathrm{M}_{\beta}(f)(x)=\sup\limits_{r>0} \frac{1}{|B(x,r)|^{1-\frac{\beta}{n}}}\int_{|y-x|\leq r}|f(y)|dy.$$
From \cite{Ste}, we know that the operator $\mathrm{M}_{\beta}$ is bounded from $L^p(\mathbb R^{n})$ to  $L^q(\mathbb R^{n})$ for $0<\beta<n$, $1<p<\frac{\beta}{n}$ and
$\frac{1}{p}-\frac{1}{q}=\frac{\beta}{n}$. Therefore, the fractional
Hardy operator $\mathrm{H}_{\beta}$ has also this property. Furthermore, Lu et al. derived the following result.

{\bf Theorem I} (\cite{LYZ}) Suppose that $0<\beta<n$, $1<p, q<\infty$ and $\frac{1}{p}-\frac{1}{q}=\frac{\beta}{n}$, then
$$\|\mathrm{H}_{\beta}(f)\|_{L^{q}(\mathbb R^{n})}\leq C \|f\|_{L^{p}(\mathbb R^{n})},$$
where,
\begin{align*}
\Big(\frac{p}{q}\Big)^{1/q}\bigg(\frac{p}{p-1}\bigg)^{{1}{q}}
\Big(\frac{q}{q-1}\Big)^{1-1/q}\Big(1-\frac{p}{q}\Big)^{1/p-1/q}\leq C\leq (\frac{p}{p-1}\big)^{\frac{p}{q}}.
\end{align*}

In \cite{LT}, the authors obtained the following the mixed radial-angular estimate for the operator $\mathrm{H}_{\beta}$.

{\bf Theorem J}
Suppose that $0<\beta<n$, $1<p_1, q_1, p_2, q_2<\infty$ and $\frac{1}{p_1}-\frac{1}{p_2}=\frac{\beta}{n}$, then
$$\|\mathrm{H}_{\beta}(f)\|_{L^{p_2}_{rad}L^{q_2}_{ang}(\mathbb R^{n})}\leq C \|f\|_{L^{p_1}_{rad}L^{q_1}_{ang}(\mathbb R^{n})},$$
Moreover, the constant $C$ satisfies the following inequality
\begin{align*}
&{{\omega}_{n}}^{1/q_2-1/q_1+\beta/n}\Big(\frac{p_1}{p_2}\Big)^{1/p_2}\bigg(\frac{p_1}{p_1-1}\bigg)^{{1}/{p_2}}
\Big(\frac{p_2}{p_2-1}\Big)^{1-1/p_2}\Big(1-\frac{p_1}{p_2}\Big)^{1/p_1-1/p_2}\leq C\\ &\leq {{\omega}_{n}}^{1/q_2-1/q_1+\beta/n}(\frac{p_1}{p_1-1}\big)^{\frac{p_1}{p_2}}.
\end{align*}

In addition, Zhao et al. in \cite{ZL} obtained the best bound for the operator $\mathrm{H}_{\beta}$.

{\bf Theorem K} (\cite{LYZ}) Suppose that $0<\beta<n$, $1<p, q<\infty$ and $\frac{1}{p}-\frac{1}{q}=\frac{\beta}{n}$, then we have
$$\|\mathrm{H}_{\beta}\|_{L^{p}(\mathbb R^{n})\rightarrow{L^{q}(\mathbb R^{n})}}=\Big(\frac{p'}{q}\Big)^{1/q}\Big[\frac{n}{q\beta}B(\frac{n}{q\beta},\frac{n}{q'\beta})\Big]^{-\beta/n}.$$

By virtue of this the best bound, we can also extend Theorem K to the case of the mixed radial-angular norms, thus we further improves  Theorem J.

\begin{theorem}\label{th4.1}
Suppose that $0<\beta<n$, $1<p_1, q_1, p_2, q_2<\infty$ and $\frac{1}{p_1}-\frac{1}{p_2}=\frac{\beta}{n}$, then we have
$$\|\mathrm{H}_{\beta}\|_{{L^{p_1}_{rad}L^{q_1}_{ang}(\mathbb R^{n})}\rightarrow L^{p_2}_{rad}L^{q_2}_{ang}(\mathbb R^{n})}={{\omega}_{n}}^{1/q_2-1/q_1+\beta/n}\Big(\frac{p'_1}{p_2}\Big)^{1/p_2}\Big[\frac{n}{p_2\beta}B(\frac{n}{p_2\beta},\frac{n}{p_2'\beta})\Big]^{-\beta/n}.$$
\end{theorem}
\begin{proof}
 By using the same method as Theorem \ref{th3.1}, we know that the norm of the operator $\mathrm{H}_{\beta}$ is equal to the norm of the operator of $\mathrm{H}_{\beta}$ restricts to the set of nonnegative radial functions, which combined with Theorem K deduces that Theorem \ref{th4.1}.
\end{proof}

Next, we will recall the weak-type estimate for the fractional Hardy operator $\mathrm{H}_{\beta}$. In 2013, Lu et al. established the following sharp estimate.

{\bf Theorem L} (\cite{LYZ}) Suppose that $0<\beta<n$, $1<p, q<\infty$ and $\frac{1}{p}-\frac{1}{q}=\frac{\beta}{n}$, then for $\lambda>0$, we have
$$\Big|\Big\{x\in\mathbb R^n: |\mathrm{H}_{\beta}(f)(x)|>\lambda \Big\}\Big|\leq \bigg(\frac{\|f\|_{L^{p}(\mathbb R^{n})}}{\lambda}\bigg)^{\frac{n}{n-\beta}}.$$
Moreover,
$$\|\mathrm{H}_{\beta}\|_{L^1(\mathbb R^n)\rightarrow L^{\frac{n}{n-\beta},\infty}(\mathbb R^n)}=1.$$
Here, the norm $\|f\|_{L^{p,\infty}(\mathbb R^n)}$ (weak $L^p(\mathbb R^n)$ norm), defined by
$$\|f\|_{L^{p,\infty}(\mathbb R^n)}=\sup\limits_{\lambda>0}\lambda\Big|\Big\{x\in\mathbb R^n: |f(x)|>\lambda \Big\}\Big|^{1/p}.$$
It can be rewrite as
\begin{align*}
\|f\|_{L^{p,\infty}(\mathbb R^n)}&=\sup\limits_{\lambda>0}\lambda\Big(\int_{\mathbb R^n}\chi_{\{x\in\mathbb R^n: |f(x)|>\lambda\}}(x)dx\Big)^{1/p}\\
&=\sup\limits_{\lambda>0}\lambda\big\|\chi_{\{\cdot\:\in\mathbb R^n\::\: |f(\cdot)|>\lambda\}}\big\|_{{L^p(\mathbb R^n)}}.
\end{align*}
Inspired by this, the first author and the third author in \cite{LT} give the definition of the weak mixed radial-angular norms  as follows:
\begin{align*}
\|f\|_{\mathcal{W}L^{p}_{rad}L^{\tilde{p}}_{ang}(\mathbb R^n)}=\sup\limits_{\lambda>0}\lambda\big\|\chi_{\{\cdot\:\in\mathbb R^n\:: \: |f(\cdot)|>\lambda\}}\big\|_{L^{p}_{rad}L^{\tilde{p}}_{ang}(\mathbb R^n)}, \quad 1\leq p, \tilde{p}<\infty,
\end{align*}
and they also extended Theorem L to the following case of  weak mixed radial-angular norms.

{\bf Theorem M}  Suppose that $0<\beta<n$, $1\leq p_1, q_1, p_2, q_2<\infty$ and $\frac{1}{p_1}-\frac{1}{p_2}=\frac{\beta}{n}$, then
the fractional Hardy operator $\mathrm{H}_{\beta}$ is bounded from ${L^{p_1}_{rad}L^{q_1}_{ang}(\mathbb R^{n})}$ to ${\mathcal{W}L^{p_2}_{rad}L^{q_2}_{ang}(\mathbb R^{n})}$.
Moreover,
$$\|\mathrm{H}_{\beta}\|_{L^{p_1}_{rad}L^{q_1}_{ang}(\mathbb R^{n})\rightarrow \mathcal{W}L^{p_2}_{rad}L^{q_2}_{ang}(\mathbb R^{n})}={{\omega}_{n}}^{1/q_2-1/q_1+\beta/n}.$$

In \cite{YL}, Yu et al. obtained the following sharp weak bound for the operator  $\mathrm{H}_{\beta}$.

{\bf Theorem N}  Let  $1<p<\infty$, $1\leq q<\infty$, $\alpha_1<np-n$, $n+\alpha_2>0$, $0<\beta\leq\frac{\alpha_1}{p-1}$. If
 $$\frac{\alpha_2+n}{q}=\frac{\alpha_1+n}{p}-\beta.$$
Then
$$\|\mathrm{H}_{\beta}\|_{L^p(\mathbb R^n,|x|^{\alpha_1})\rightarrow L^{q,\infty}(\mathbb R^n,|x|^{\alpha_2})}=\frac{1}{\nu_{n}^{1-\frac{\beta}{n}}}\Big(\frac{\omega_n}{n-\frac{\alpha_1}{p-1}}\Big)^{1/p'}\Big(\frac{\omega_n}{n+\alpha_2}\Big)^{1/q}.$$

Similarly, we define the weak weighted mixed radial-angular norms as follows:
\begin{align*}
&\|f\|_{\mathcal{W}L^{p}_{rad}L^{\tilde{p}}_{ang}(\mathbb R^n,\omega)}\\
&=\sup\limits_{\lambda>0}\lambda\big\|\chi_{\{\cdot\:\in\mathbb R^n\:: \: |f(\cdot)|>\lambda\}}\big\|_{L^{p}_{rad}L^{\tilde{p}}_{ang}(\mathbb R^n,\omega)} \quad (1\leq p, \tilde{p}<\infty)\\
&=\sup\limits_{\lambda>0}\lambda\Big(\int_{0}^{\infty}\Big(\int_{S^{n-1}}|\chi_{\{r\theta\in(0,\infty)\times{S}^{n-1}: |f(r\theta)|>\lambda\}}|^{\tilde{p}}d\sigma(\theta)\Big)^{p/\tilde{p}}\omega(r)r^{n-1}dr\Big)^{1/p}.
\end{align*}

Next, we will extended Theorem N to the following case of  weak  weighted mixed radial-angular norms.

\begin{theorem}\label{th4.2} Let  $1<p_1, \tilde{p}_1, \tilde{p}_2<\infty$, $1\leq p_2<\infty$, $\alpha_1<np_1-n$, $n+\alpha_2>0$, $0<\beta\leq\frac{\alpha_1}{p_1-1}$. If
 $$\frac{\alpha_2+n}{p_2}=\frac{\alpha_1+n}{p_1}-\beta.$$
Then
$$\|\mathrm{H}_{\beta}\|_{L^{p_1}_{rad}L^{\tilde{p_1}}_{ang}(\mathbb R^{n},|x|^{\alpha_1})\rightarrow \mathcal{W}L^{p_2}_{rad}L^{\tilde{p_2}}_{ang}(\mathbb R^{n},|x|^{\alpha_2})}=\frac{\omega_{n}^{\frac{1}{\tilde{p_1}'}+\frac{1}{\tilde{p_2}}}}{\nu_{n}^{1-\frac{\beta}{n}}}\Big(\frac{1}{n-\frac{\alpha_1}{p_1-1}}\Big)^{1/p'_1}\Big(\frac{1}{n+\alpha_2}\Big)^{1/p_2}.$$
\end{theorem}

\begin{proof}
Note that $n-\frac{\alpha_1}{p_1-1}>0$, we have
\begin{align*}
&\frac{1}{|B(0,|x|)|^{1-\frac{\beta}{n}}}\Big|\int_{|y|<|x|}f(y)dy\Big|\\
&\leq\frac{\omega_{n}^{\frac{1}{\tilde{p_1}'}}}{(\nu_{n}|x|^n)^{1-\frac{\beta}{n}}}\int_{0}^{|x|}\rho^{-\alpha_1/p_1}\Big(\int_{{S}^{n-1}}|f(\rho,\theta)|^{\tilde{p_1}}d\sigma(\theta)\Big)^{1/\tilde{p_1}}\rho^{\alpha_1/p_1}\rho^{n-1}d\rho\\
&\leq \frac{\omega_{n}^{\frac{1}{\tilde{p_1}'}}}{(\nu_{n}|x|^n)^{1-\frac{\beta}{n}}}\Big(\int_{0}^{|x|}\rho^{-\alpha_1p'_1/p_1}\rho^{n-1}d\rho\Big)^{1/p'_1}\\
&\quad\times\bigg(\int_{0}^{\infty}\Big(\int_{{S}^{n-1}}|f(\rho,\theta)|^{\tilde{p_1}}d\sigma(\theta)\Big)^{p_1/\tilde{p_1}} \rho^{\alpha_1}\rho^{n-1}d\rho\Big)^{1/p_1}\\
&=\frac{\omega_{n}^{\frac{1}{\tilde{p_1}'}}}{\nu_{n}^{1-\frac{\beta}{n}}}\Big(\frac{1}{n-\frac{\alpha_1}{p_1-1}}\Big)^{1/p'_1}|x|^{-\frac{n}{p_1}-\frac{\alpha_1}{p_1}+\beta}\|f\|_{L^{p_1}_{rad}L^{\tilde{p_1}}_{ang}(\mathbb R^{n},|x|^{\alpha_1})}\\
&:=I_0|x|^{-\frac{n}{p_1}-\frac{\alpha_1}{p_1}+\beta}.
\end{align*}
Thus,
\begin{align*}
&\big\|\chi_{\{x\in\mathbb R^n:  |\mathrm{H}_{\beta}f(x)|>\lambda\}}\big\|_{L^{p_2}_{rad}L^{\tilde{p_2}}_{ang}(\mathbb R^n,|x|^{\alpha_2})}\\
&\leq\Big\|\chi_{\{x\in\mathbb R^n: I_0|x|^{-\frac{n}{p_1}-\frac{\alpha_1}{p_1}+\beta}>\lambda\}}\Big\|_{L^{p_2}_{rad}L^{\tilde{p_2}}_{ang}(\mathbb R^n,|x|^{\alpha_2})}\\
&=\Big\|\chi_{\{x\in\mathbb R^n: {|x|^{\frac{n+\alpha_2}{p_2}}}<\frac{I_0}{\lambda}\}}\Big\|_{L^{p_2}_{rad}L^{\tilde{p_2}}_{ang}(\mathbb R^n,|x|^{\alpha_2})}\\
&={{\omega}_{n}}^{1/\tilde{p_2}}\Bigg(\int_{0}^{\big(\frac{I_0}{\lambda}\big)^{\frac{p_2}{n+\alpha_2}}}r^{\alpha_2+n-1}dr\Bigg)^{1/p_2}\\
&=\frac{\omega_{n}^{\frac{1}{\tilde{p_1}'}+\frac{1}{\tilde{p_2}}}}{\nu_{n}^{1-\frac{\beta}{n}}}\Big(\frac{1}{n-\frac{\alpha_1}{p_1-1}}\Big)^{1/p'_1}\Big(\frac{1}{n+\alpha_2}\Big)^{1/p_2}\frac{1}{\lambda}\|f\|_{L^{p_1}_{rad}L^{\tilde{p_1}}_{ang}(\mathbb R^{n},|x|^{\alpha_1})}.
\end{align*}
Next, we will show that the constant
 $$\frac{\omega_{n}^{\frac{1}{\tilde{p_1}'}+\frac{1}{\tilde{p_2}}}}{\nu_{n}^{1-\frac{\beta}{n}}}\Big(\frac{1}{n-\frac{\alpha_1}{p_1-1}}\Big)^{1/p'_1}\Big(\frac{1}{n+\alpha_2}\Big)^{1/p_2}$$ is sharp.
For this purpose, we set the function
$$f_0(x)=|x|^{-\frac{\alpha_1}{p_1-1}}\chi_{\{|x|<1\}}(x),$$
then we have
$$\|f_0\|_{L^{p_1}_{rad}L^{\tilde{p_1}}_{ang}(\mathbb R^{n},|x|^{\alpha_1})}={{\omega}_{n}}^{1/\tilde{p_1}}\Big(\frac{1}{n-\frac{\alpha_1}{p_1-1}}\Big)^{1/p_1}.$$
It follows from \cite{YL} that
\begin{eqnarray*}
\mathrm{H}_{\beta}(f_0)(x)=\frac{1}{\nu_{n}^{1-\frac{\beta}{n}}}\frac{\omega_n}{n-\frac{\alpha_1}{p_1-1}}
\begin{cases}
|x|^{\beta-\frac{\alpha_1}{p_1-1}},&\mbox{}|x|<1,\\
|x|^{\beta-n},&\mbox{}|x|\geq1.
\end{cases}
\end{eqnarray*}
For convenience, denote $I_1=\frac{1}{\nu_{n}^{1-\frac{\beta}{n}}}\frac{\omega_n}{n-\frac{\alpha_1}{p_1-1}}$.
Hence,
$$\{x\in\mathbb R^{n}: |\mathrm{H}_{\beta}(f_0)(x)|>\lambda\}=\{|x|<1: I_1|x|^{\beta-\frac{\alpha_1}{p_1-1}}>\lambda\}\cup\{|x|\geq1: I_1|x|^{\beta-n}>\lambda\}.$$
If $0<\lambda<I_1$, note that $\beta\leq\frac{\alpha_1}{p-1}$ and $\alpha_1<np_1-n$, we have $\beta<n$, and
\begin{align*}
\Big\{x\in\mathbb R^{n}: |\mathrm{H}_{\beta}(f_0)(x)|>\lambda\Big\}&=\{|x|<1\}\cup\Big\{|x|\geq1: |x|<\Big(\frac{I_1}{\lambda}\Big)^{\frac{1}{n-\beta}}\Big\}\\
&=\Big\{x\in\mathbb R^{n}: |x|<\Big(\frac{I_1}{\lambda}\Big)^{\frac{1}{n-\beta}}\Big\}.
\end{align*}
By using the facts
$$\|f_0\|_{L^{p_1}_{rad}L^{\tilde{p_1}}_{ang}(\mathbb R^{n},|x|^{\alpha_1})}={{\omega}_{n}}^{1/\tilde{p_1}}\Big(\frac{1}{n-\frac{\alpha_1}{p_1-1}}\Big)^{1/p_1}, \alpha_2+n>0, 1-\frac{n+\alpha_2}{(n-\beta)p_2}>0,$$
we obtain that
\begin{align*}
&\sup\limits_{0<\lambda<I_1}\lambda\big\|\chi_{\{x\in\mathbb R^n:  |\mathrm{H}_{\beta}f_0(x)|>\lambda\}}\big\|_{L^{p_2}_{rad}L^{\tilde{p_2}}_{ang}(\mathbb R^n,|x|^{\alpha_2})}\\
&\leq\sup\limits_{0<\lambda<I_1}\lambda\Bigg\|\chi_{\{x\in\mathbb R^n: |x|<\Big(\frac{I_1}{\lambda}\Big)^{\frac{1}{n-\beta}}\}}\Bigg\|_{L^{p_2}_{rad}L^{\tilde{p_2}}_{ang}(\mathbb R^n,|x|^{\alpha_2})}\\
&=\sup\limits_{0<\lambda<I_1}\lambda{{\omega}_{n}}^{1/\tilde{p_2}}\Bigg(\int_{0}^{\big(\frac{I_1}{\lambda}\big)^{\frac{1}{n-\beta}}}r^{\alpha_2+n-1}dr\Bigg)^{1/p_2}\\
&=\sup\limits_{0<\lambda<I_1}{{\omega}_{n}}^{1/\tilde{p_2}}\Big(\frac{1}{n+\alpha_2}\Big)^{1/p_2}I_1^{\frac{n+\alpha_2}{p_2(n-\beta)}}\lambda^{1-\frac{n+\alpha_2}{p_2(n-\beta)}}\\
&={{\omega}_{n}}^{1/\tilde{p_2}}\Big(\frac{1}{n+\alpha_2}\Big)^{1/p_2}I_1\\
&={{\omega}_{n}}^{1/\tilde{p_2}}\Big(\frac{1}{n+\alpha_2}\Big)^{1/p_2}\frac{1}{\nu_{n}^{1-\frac{\beta}{n}}}\frac{\omega_n}{n-\frac{\alpha_1}{p_1-1}}\\
&=\frac{\omega_{n}^{\frac{1}{\tilde{p_1}'}+\frac{1}{\tilde{p_2}}}}{\nu_{n}^{1-\frac{\beta}{n}}}\Big(\frac{1}{n-\frac{\alpha_1}{p_1-1}}\Big)^{1/p'_1}\Big(\frac{1}{n+\alpha_2}\Big)^{1/p_2}\|f_0\|_{L^{p_1}_{rad}L^{\tilde{p_1}}_{ang}(\mathbb R^{n},|x|^{\alpha_1})}.
\end{align*}
If $\lambda\geq I_1$ and $\beta=\frac{\alpha_1}{p_1-1}$, we have
$$\Big\{x\in\mathbb R^{n}: |\mathrm{H}_{\beta}(f_0)(x)|>\lambda\Big\}=\emptyset.$$
If $\lambda\geq I_1$ and $\beta<\frac{\alpha_1}{p_1-1}$, note that $\alpha_1<np_1-n$,  we have  $\beta<n$, and
$$\Big\{x\in\mathbb R^{n}: |\mathrm{H}_{\beta}(f_0)(x)|>\lambda\Big\}=\Big\{x\in\mathbb R^{n}: |x|<\Big(\frac{I_1}{\lambda}\Big)^{\frac{1}{\frac{\alpha_1}{p-1}-\beta}}\Big\}.$$
By using the following facts
$$\|f_0\|_{L^{p_1}_{rad}L^{\tilde{p_1}}_{ang}(\mathbb R^{n},|x|^{\alpha_1})}={{\omega}_{n}}^{1/\tilde{p_1}}\Big(\frac{1}{n-\frac{\alpha_1}{p_1-1}}\Big)^{1/p_1}, \alpha_2+n>0, 1-\frac{n+\alpha_2}{\frac{1}{(\frac{\alpha_1}{p-1}-\beta})p_2}<0,$$
we derive that
\begin{align*}
&\sup\limits_{\lambda\geq I_1}\lambda\big\|\chi_{\{x\in\mathbb R^n:  |\mathrm{H}_{\beta}f_0(x)|>\lambda\}}\big\|_{L^{p_2}_{rad}L^{\tilde{p_2}}_{ang}(\mathbb R^n,|x|^{\alpha_2})}\\
&\leq\sup\limits_{\lambda\geq I_1}\lambda\Bigg\|\chi_{\{x\in\mathbb R^n: |x|<\Big(\frac{I_1}{\lambda}\Big)^{\frac{1}{\frac{\alpha_1}{p-1}-\beta}}\}}\Bigg\|_{L^{p_2}_{rad}L^{\tilde{p_2}}_{ang}(\mathbb R^n,|x|^{\alpha_2})}\\
&=\sup\limits_{\lambda\geq L}\lambda{{\omega}_{n}}^{1/\tilde{p_2}}\Bigg(\int_{0}^{\big(\frac{I_1}{\lambda}\big)^{\frac{1}{\frac{\alpha_1}{p-1}-\beta}}}r^{\alpha_2+n-1}dr\Bigg)^{1/p_2}\\
&=\sup\limits_{\lambda\geq L}{{\omega}_{n}}^{1/\tilde{p_2}}\Big(\frac{1}{n+\alpha_2}\Big)^{1/p_2}I_1^{\frac{n+\alpha_2}{p_2\big({\frac{1}{\frac{\alpha_1}{p-1}-\beta}}\big)}}\lambda^{1-\frac{n+\alpha_2}{p_2\big({\frac{1}{\frac{\alpha_1}{p-1}-\beta}}\big)}}\\
&={{\omega}_{n}}^{1/\tilde{p_2}}\Big(\frac{1}{n+\alpha_2}\Big)^{1/p_2}I_1\\
&=\frac{\omega_{n}^{\frac{1}{\tilde{p_1}'}+\frac{1}{\tilde{p_2}}}}{\nu_{n}^{1-\frac{\beta}{n}}}\Big(\frac{1}{n-\frac{\alpha_1}{p_1-1}}\Big)^{1/p'_1}\Big(\frac{1}{n+\alpha_2}\Big)^{1/p_2}\|f_0\|_{L^{p_1}_{rad}L^{\tilde{p_1}}_{ang}(\mathbb R^{n},|x|^{\alpha_1})}.
\end{align*}
Therefore, we have
$$\|\mathrm{H}_{\beta}\|_{L^{p_1}_{rad}L^{\tilde{p_1}}_{ang}(\mathbb R^{n},|x|^{\alpha_1})\rightarrow \mathcal{W}L^{p_2}_{rad}L^{\tilde{p_2}}_{ang}(\mathbb R^{n},|x|^{\alpha_2})}=\frac{\omega_{n}^{\frac{1}{\tilde{p_1}'}+\frac{1}{\tilde{p_2}}}}{\nu_{n}^{1-\frac{\beta}{n}}}\Big(\frac{1}{n-\frac{\alpha_1}{p_1-1}}\Big)^{1/p'_1}\Big(\frac{1}{n+\alpha_2}\Big)^{1/p_2}.$$
The proof of Theorem \ref{th4.2} is completed.
\end{proof}

For the case of endpoint, the following of the weak  weighted mixed radial-angular sharp estimate will be obtained.
\begin{theorem}\label{th4.3} Let  $1<\tilde{p}_1, \tilde{p}_2<\infty$,  $n+\alpha_2>0$,  $0<\beta<n$.
Then
$$\|\mathrm{H}_{\beta}\|_{L^{1}_{rad}L^{\tilde{p_1}}_{ang}(\mathbb R^{n})\rightarrow \mathcal{W}L^{\frac{n+\alpha_2}{n-\beta}}_{rad}L^{\tilde{p_2}}_{ang}(\mathbb R^{n},|x|^{\alpha_2})}=\frac{\omega_{n}^{\frac{1}{\tilde{p_1}'}+\frac{1}{\tilde{p_2}}}}{\nu_{n}^{1-\frac{\beta}{n}}}\Big(\frac{1}{n+\alpha_2}\Big)^{(n-\beta)/(n+\alpha_2)}.$$
\end{theorem}
\begin{proof}
It is easy to see that
$$|\mathrm{H}_{\beta}f(x)|\leq\frac{\omega_{n}^{\frac{1}{\tilde{p_1}'}}}{\nu_{n}^{1-\frac{\beta}{n}}}\frac{1}{|x|^{n-\beta}}\|f\|_{L^{1}_{rad}L^{\tilde{p_1}}_{ang}(\mathbb R^{n})}:=K_0\frac{1}{|x|^{n-\beta}}.$$
 Since $n-\beta>0$ and $n+\alpha_2>0$, we have
 \begin{align*}
&\sup\limits_{\lambda>0}\lambda\big\|\chi_{\{x\in\mathbb R^n:  |\mathrm{H}_{\beta}f(x)|>\lambda\}}\big\|_{L^{\frac{n+\alpha_2}{n-\beta}}_{rad}L^{\tilde{p_2}}_{ang}(\mathbb R^n,|x|^{\alpha_2})}\\
&\leq\sup\limits_{\lambda>0}\lambda\Bigg\|\chi_{\{x\in\mathbb R^n: |x|<\Big(\frac{K_0}{\lambda}\Big)^{\frac{1}{n-\beta}}\}}\Bigg\|_{L^{\frac{n+\alpha_2}{n-\beta}}_{rad}L^{\tilde{p_2}}_{ang}(\mathbb R^n,|x|^{\alpha_2})}\\
&=\sup\limits_{\lambda>0}\lambda{{\omega}_{n}}^{1/\tilde{p_2}}\Bigg(\int_{0}^{\big(\frac{K_0}{\lambda}\big)^{\frac{1}{n-\beta}}}r^{\alpha_2+n-1}dr\Bigg)^{^{(n-\beta)/(n+\alpha_2)}}\\
\end{align*}
\begin{align*}
&={{\omega}_{n}}^{1/\tilde{p_2}}\Big(\frac{1}{n+\alpha_2}\Big)^{(n-\beta)/(n+\alpha_2)}K_0\\
&=\frac{\omega_{n}^{\frac{1}{\tilde{p_1}'}+\frac{1}{\tilde{p_2}}}}{\nu_{n}^{1-\frac{\beta}{n}}}\Big(\frac{1}{n+\alpha_2}\Big)^{(n-\beta)/(n+\alpha_2)}\|f\|_{L^{1}_{rad}L^{\tilde{p_1}}_{ang}(\mathbb R^{n})}.
\end{align*}
On the other hand, we consider the function
$$f_0(x)=\chi_{\{ x\in\mathbb R^n: |x|<1\}}(x),$$
then we have
$$\|f_0\|_{L^{1}_{rad}L^{\tilde{p_1}}_{ang}(\mathbb R^{n})}=\frac{{{\omega}_{n}}^{1/\tilde{p_1}}}{n}.$$
It follows from \cite{YL} that
\begin{eqnarray*}
\mathrm{H}_{\beta}(f_0)(x)=\frac{1}{\nu_{n}^{1-\frac{\beta}{n}}}\frac{{\omega}_{n}}{n}
\begin{cases}
|x|^{\beta},&\mbox{}|x|<1,\\
|x|^{\beta-n},&\mbox{}|x|\geq1.
\end{cases}
\end{eqnarray*}
Denote $M_0=\frac{1}{\nu_{n}^{1-\frac{\beta}{n}}}\frac{\omega_{n}}{n}$ and
$$\{x\in\mathbb R^{n}: |\mathrm{H}_{\beta}(f_0)(x)|>\lambda\}=\{|x|<1: M_0|x|^{\beta}>\lambda\}\cup\{|x|\geq1: M_0|x|^{\beta-n}>\lambda\}.$$
If $0<\lambda<M_0$, note that $0<\beta<n$, we have
\begin{align*}
\Big\{x\in\mathbb R^{n}: |\mathrm{H}_{\beta}(f_0)(x)|>\lambda\Big\}=\Big\{x\in\mathbb R^{n}: \Big(\frac{\lambda}{M_0}\Big)^{\frac{1}{\beta}} <|x|<\Big(\frac{M_0}{\lambda}\Big)^{\frac{1}{n-\beta}}\Big\}.
\end{align*}
By using the following facts
$$\|f_0\|_{L^{1}_{rad}L^{\tilde{p_1}}_{ang}(\mathbb R^{n})}=\frac{{{\omega}_{n}}^{1/\tilde{p_1}}}{n}, \alpha_2+n>0,  0<\beta<n,$$
we obtain that

\begin{align*}
&\sup\limits_{0<\lambda<M_0}\lambda\big\|\chi_{\{x\in\mathbb R^n:  |\mathrm{H}_{\beta}f_0(x)|>\lambda\}}\big\|_{L^{\frac{n+\alpha_2}{n-\beta}}_{rad}L^{\tilde{p_2}}_{ang}(\mathbb R^n,|x|^{\alpha_2})}\\
&\leq\sup\limits_{0<\lambda<M_0}\lambda\Bigg\|\chi_{\{x\in\mathbb R^n: \Big(\frac{\lambda}{M_0}\Big)^{\frac{1}{\beta}} <|x|<\Big(\frac{M_0}{\lambda}\Big)^{\frac{1}{n-\beta}}\}}\Bigg\|_{L^{\frac{n+\alpha_2}{n-\beta}}_{rad}L^{\tilde{p_2}}_{ang}(\mathbb R^n,|x|^{\alpha_2})}\\
&=\sup\limits_{0<\lambda<M_0}\lambda{{\omega}_{n}}^{1/\tilde{p_2}}\Bigg(\int_{\Big(\frac{\lambda}{M_0}\Big)^{\frac{1}{\beta}}}^{\big(\frac{M_0}{\lambda}\big)^{\frac{1}{n-\beta}}}r^{\alpha_2+n-1}dr\Bigg)^{^{(n-\beta)/(n+\alpha_2)}}\\
&={{\omega}_{n}}^{1/\tilde{p_2}}\Big(\frac{1}{n+\alpha_2}\Big)^{(n-\beta)/(n+\alpha_2)}M_0\\
&={{\omega}_{n}}^{1/\tilde{p_2}}\Big(\frac{1}{n+\alpha_2}\Big)^{(n-\beta)/(n+\alpha_2)}\frac{1}{\nu_{n}^{1-\frac{\beta}{n}}}\frac{\omega_{n}}{n}\\
&=\frac{\omega_{n}^{\frac{1}{\tilde{p_1}'}+\frac{1}{\tilde{p_2}}}}{\nu_{n}^{1-\frac{\beta}{n}}}\Big(\frac{1}{n+\alpha_2}\Big)^{(n-\beta)/(n+\alpha_2)}\|f_0\|_{L^{1}_{rad}L^{\tilde{p_1}}_{ang}(\mathbb R^{n})}.
\end{align*}

If $\lambda\geq M_0$ and by using the condition  $0<\beta<n$, we have
$$\Big\{x\in\mathbb R^{n}: |\mathrm{H}_{\beta}(f_0)(x)|>\lambda\Big\}=\emptyset.$$

Therefore, we have
$$\|\mathrm{H}_{\beta}\|_{L^{1}_{rad}L^{\tilde{p_1}}_{ang}(\mathbb R^{n})\rightarrow \mathcal{W}L^{\frac{n+\alpha_2}{n-\beta}}_{rad}L^{\tilde{p_2}}_{ang}(\mathbb R^{n},|x|^{\alpha_2})}=\frac{\omega_{n}^{\frac{1}{\tilde{p_1}'}+\frac{1}{\tilde{p_2}}}}{\nu_{n}^{1-\frac{\beta}{n}}}\Big(\frac{1}{n+\alpha_2}\Big)^{(n-\beta)/(n+\alpha_2)}.$$
The proof of Theorem \ref{th4.3} is completed.
\end{proof}

The conjugate operator $\mathrm{H}_{\beta}^*$ of the $n$-dimensional fractional
Hardy operator $\mathrm{H}_{\beta}$ is also an important operator stated as
$$\mathrm{H}^*_{\beta}(f)(x)= \int_{|y|\geq|x|}\frac{f(y)}{|B(0,|y|)|^{1-\frac{\beta}{n}}}dy.$$

In \cite{GHC}, Gao et al. obtained the following sharp weak bound for the operator  $\mathrm{H}^*_{\beta}$.

{\bf Theorem O}  Let  $0\leq\beta<n$, $1<p<\frac{n+\alpha}{\beta}$, and
 $$\frac{\alpha_1+n}{q}=\frac{\alpha+n}{p}-\beta.$$
 Then
$$\|\mathrm{H}^*_{\beta}\|_{L^p(\mathbb R^n,|x|^{\alpha})\rightarrow L^{q,\infty}(\mathbb R^n,|x|^{\alpha_1})}=\nu_n^{1/n(\frac{\alpha}{p}-\frac{\alpha_1}{q})}\Big(\frac{n}{n+\alpha_1}\Big)^{1/p'+1/q}\Big(\frac{q}{p'}\Big)^{1/p'}.$$

Similarly, we will extended Theorem O to the following case of  weak  weighted mixed radial-angular norms.

\begin{theorem}\label{th4.4} Let  $1<\tilde{p}_1, \tilde{p}_2<\infty$, $1< p_1<\frac{n+\alpha_1}{\beta}$. If
 $$\frac{\alpha_1+n}{p_2}=\frac{\alpha+n}{p_1}-\beta.$$
 Then
$$\|\mathrm{H}_{\beta}^*\|_{L^{p_1}_{rad}L^{\tilde{p_1}}_{ang}(\mathbb R^{n},|x|^{\alpha})\rightarrow \mathcal{W}L^{p_2}_{rad}L^{\tilde{p_2}}_{ang}(\mathbb R^{n},|x|^{\alpha_1})}=\frac{\omega_{n}^{\frac{1}{\tilde{p_1}'}+\frac{1}{\tilde{p_2}}}}{\nu_{n}^{1-\frac{\beta}{n}}}\Big(\frac{1}{n+\alpha_1}\Big)^{1/p'_1+1/p_2}\Big(\frac{p_2}{p'_1}\Big)^{1/p'_1}.$$
\end{theorem}

\begin{proof}
Applying the H\"{o}lder inequality and  $1< p_1<\frac{n+\alpha_1}{\beta}$, we have
\begin{align*}
|\mathrm{H}_{\beta}^*f(x)|&\leq\int_{|y|\geq|x|}|f(y)|B(0,|y|)|^{\frac{\beta}{n}-1}dy\\
&=\frac{1}{\nu_{n}^{1-\frac{\beta}{n}}}\int_{|y|\geq|x|}|f(y)||y|^{{\beta-n}}dy\\
&=\frac{\omega_{n}^{\frac{1}{\tilde{p_1}'}}}{\nu_{n}^{1-\frac{\beta}{n}}}\int_{|x|}^{\infty}\Big(\int_{{S}^{n-1}}|f(\rho\theta)|^{\tilde{p_1}}d\sigma(\theta)\Big)^{1/\tilde{p_1}}\rho^{\beta-n}\rho^{n-1}d\rho\\
&\leq\frac{\omega_{n}^{\frac{1}{\tilde{p_1}'}}}{\nu_{n}^{1-\frac{\beta}{n}}}\Big(\int_{|x|}^{\infty}\Big(\int_{{S}^{n-1}}|f(\rho\theta)|^{\tilde{p_1}}d\sigma(\theta)\Big)^{p_1/\tilde{p_1}}\rho^{\frac{\alpha}{p_1}}\rho^{n-1}d\rho\Big)^{1/p_1}\\
&\quad\times\Big(\int_{|x|}^{\infty}\rho^{-\frac{\alpha  p'_1}{p_1}}\rho^{({\beta-n})p'_1}\rho^{n-1}d\rho\Big)^{1/p'_1}\\
&\leq\frac{\omega_{n}^{\frac{1}{\tilde{p_1}'}}}{\nu_{n}^{1-\frac{\beta}{n}}}\Big(\frac{p_1-1}{n+\alpha-p_1\beta}\Big)^{1/p'_1}|x|^{\beta-\frac{n+\alpha}{p_1}}\|f\|_{L^{p_1}_{rad}L^{\tilde{p_1}}_{ang}(\mathbb R^{n},|x|^{\alpha})}.
\end{align*}
By using the condition  $\frac{\alpha_2+n}{p_2}=\frac{\alpha_1+n}{p_1}-\beta$, we derive that
$$|\mathrm{H}_{\beta}^*f(x)|\leq\frac{\omega_{n}^{\frac{1}{\tilde{p_1}'}}}{\nu_{n}^{1-\frac{\beta}{n}}}\Big(\frac{p_2}{p'_1(n+\alpha_1)}\Big)^{1/p'_1}|x|^{-\frac{n+\alpha_1}{p_2}}\|f\|_{L^{p_1}_{rad}L^{\tilde{p_1}}_{ang}(\mathbb R^{n},|x|^{\alpha})}:=N_0|x|^{-\frac{n+\alpha_1}{p_2}}.$$
Then for any $\lambda>0$,
$$\{x\in\mathbb R^{n}: |\mathrm{H}^*_{\beta}(f)(x)|>\lambda\}\subset\{x\in\mathbb R^{n}: |x|<\Big(\frac{N_0}{\lambda}\Big)^{\frac{p_2}{n+\alpha_1}}\}.$$
Thus, we have
\begin{align*}
&\sup\limits_{0<\lambda<M_0}\lambda\big\|\chi_{\{x\in\mathbb R^n:  |\mathrm{H}^*_{\beta}f(x)|>\lambda\}}\big\|_{L^{p_2}_{rad}L^{\tilde{p_2}}_{ang}(\mathbb R^n,|x|^{\alpha_1})}\\
&\leq\sup\limits_{\lambda>0}\lambda\Bigg\|\chi_{\{x\in\mathbb R^n: |x|<\Big(\frac{N_0}{\lambda}\Big)^{\frac{p_2}{n+\alpha_1}}\}}\Bigg\|_{L^{p_2}_{rad}L^{\tilde{p_2}}_{ang}(\mathbb R^n,|x|^{\alpha_1})}\\
&=\sup\limits_{\lambda>0}\lambda{{\omega}_{n}}^{1/\tilde{p_2}}\Bigg(\int_{0}^{\big(\frac{N_0}{\lambda}\big)^{\frac{p_2}{n+\alpha_1}}}r^{\alpha_1+n-1}dr\Bigg)^{^{1/p_2}}\\
&={{\omega}_{n}}^{1/\tilde{p_2}}\Big(\frac{1}{n+\alpha_1}\Big)^{1/p_2}N_0\\
&={{\omega}_{n}}^{1/\tilde{p_2}}\Big(\frac{1}{n+\alpha_1}\Big)^{1/p_2}\frac{\omega_{n}^{\frac{1}{\tilde{p_1}'}}}{\nu_{n}^{1-\frac{\beta}{n}}}\Big(\frac{1}{n+\alpha_1}\Big)^{1/p'_1}\Big(\frac{p_2}{p'_1}\Big)^{1/p'_1}\|f\|_{L^{p_1}_{rad}L^{\tilde{p_1}}_{ang}(\mathbb R^{n},|x|^{\alpha})}\\
&=\frac{\omega_{n}^{\frac{1}{\tilde{p_1}'}+\frac{1}{\tilde{p_2}}}}{\nu_{n}^{1-\frac{\beta}{n}}}\Big(\frac{1}{n+\alpha_1}\Big)^{1/p'_1+1/p_2}\Big(\frac{p_2}{p'_1}\Big)^{1/p'_1}\|f\|_{L^{p_1}_{rad}L^{\tilde{p_1}}_{ang}(\mathbb R^{n},|x|^{\alpha})}.
\end{align*}
On the other hand, we take the function
$$f_0(x)=|x|^{\frac{\beta-n-\alpha}{p-1}}\chi_{\{ |x|>1\}}(x),$$
and $$\|f_0\|_{L^{p_1}_{rad}L^{\tilde{p_1}}_{ang}(\mathbb R^{n},|x|^{\alpha})}={{\omega}_{n}}^{1/\tilde{p_1}}\Big(\frac{p_1-1}{n+\alpha-p_1\beta}\Big)^{1/p_1}.$$
Note that $\frac{\alpha_1+n}{p_2}=\frac{\alpha+n}{p_1}-\beta$, we get
$$\|f_0\|_{L^{p_1}_{rad}L^{\tilde{p_1}}_{ang}(\mathbb R^{n},|x|^{\alpha})}={{\omega}_{n}}^{1/\tilde{p_1}}\Big(\frac{1}{n+\alpha_1}\Big)^{1/p_1}\Big(\frac{p_2}{p'_1}\Big)^{1/p_1}.$$
In addition,
\begin{eqnarray*}
\mathrm{H}^*_{\beta}(f_0)(x)=\frac{{\omega}_{n}}{\nu_{n}^{1-\frac{\beta}{n}}}\frac{1}{n+\alpha_1}\frac{p_2}{p'_1}
\begin{cases}
1,&\mbox{}|x|<1,\\
|x|^{-\frac{p'_1(n+\alpha_1)}{p_2}},&\mbox{}|x|\geq1.
\end{cases}
\end{eqnarray*}
For convenience, denote $P_0=\frac{{\omega}_{n}}{\nu_{n}^{1-\frac{\beta}{n}}}\frac{1}{n+\alpha_1}\frac{p_2}{p'_1}$.
Hence,
$$\{x\in\mathbb R^{n}: |\mathrm{H}_{\beta}(f_0)(x)|>\lambda\}=\{|x|<1: P_0>\lambda\}\cup\{|x|\geq1: P_0|x|^{-\frac{p'_1(n+\alpha_1)}{p_2}}>\lambda\}.$$
Therefore, we derive that
\begin{align*}
&\sup\limits_{0<\lambda<P_0}\lambda\big\|\chi_{\{x\in\mathbb R^n:  |\mathrm{H}^*_{\beta}f_0(x)|>\lambda\}}\big\|_{L^{p_2}_{rad}L^{\tilde{p_2}}_{ang}(\mathbb R^n,|x|^{\alpha_1})}\\
&=\sup\limits_{0<\lambda<P_0}\lambda{{\omega}_{n}}^{1/\tilde{p_2}}\Bigg(\int_{0}^{\frac{P_0}{\lambda}}r^{{\alpha_1}+n-1}dr\Bigg)^{1/p_2}\\
&={{\omega}_{n}}^{1/\tilde{p_2}}\Big(\frac{1}{n+\alpha_1}\Big)^{1/p_2}P_0\\
&={{\omega}_{n}}^{1/\tilde{p_2}}\Big(\frac{1}{n+\alpha_1}\Big)^{1/p_1}\frac{{\omega}_{n}}{\nu_{n}^{1-\frac{\beta}{n}}}\frac{1}{n+\alpha_1}\frac{p_2}{p'_1}\\
&=\frac{\omega_{n}^{\frac{1}{\tilde{p_1}'}+\frac{1}{\tilde{p_2}}}}{\nu_{n}^{1-\frac{\beta}{n}}}\Big(\frac{1}{n-\frac{\alpha_1}{p_1-1}}\Big)^{1/p'_1}\Big(\frac{1}{n+\alpha_2}\Big)^{1/p_2}\|f_0\|_{L^{p_1}_{rad}L^{\tilde{p_1}}_{ang}(\mathbb R^{n},|x|^{\alpha})}.
\end{align*}
If $\lambda>P_0$, then
$$\Big\{x\in\mathbb R^{n}: |\mathrm{H}^*_{\beta}(f_0)(x)|>\lambda\Big\}=\emptyset.$$
Hence, we have
$$\|\mathrm{H}_{\beta}^*\|_{L^{p_1}_{rad}L^{\tilde{p_1}}_{ang}(\mathbb R^{n},|x|^{\alpha})\rightarrow \mathcal{W}L^{p_2}_{rad}L^{\tilde{p_2}}_{ang}(\mathbb R^{n},|x|^{\alpha_1})}=\frac{\omega_{n}^{\frac{1}{\tilde{p_1}'}+\frac{1}{\tilde{p_2}}}}{\nu_{n}^{1-\frac{\beta}{n}}}\Big(\frac{1}{n+\alpha_1}\Big)^{1/p'_1+1/p_2}\Big(\frac{p_2}{p'_1}\Big)^{1/p'_1}.$$
This proves the proof of Theorem \ref{th4.4}.
\end{proof}

We also have the following  weak  weighted mixed radial-angular sharp estimate.
\begin{theorem}\label{th4.5}
Let $0\leq\beta<n$,  $1<p, \tilde{p}_1, \tilde{p}_2<\infty$ and $\min\{\alpha,\alpha_1\}>-n$. If
 $$\frac{\alpha_1+n}{p}=\alpha+n-\beta.$$
 Then
$$\|\mathrm{H}^*_{\beta}\|_{L^{1}_{rad}L^{\tilde{p_1}}_{ang}(\mathbb R^{n},|x|^{\alpha})\rightarrow \mathcal{W}L^{p}_{rad}L^{\tilde{p_2}}_{ang}(\mathbb R^{n},|x|^{\alpha_2})}=\frac{\omega_{n}^{\frac{1}{\tilde{p_1}'}+\frac{1}{\tilde{p_2}}}}{\nu_{n}^{1-\frac{\beta}{n}}}\Big(\frac{1}{n+\alpha_1}\Big)^{1/p}.$$
\end{theorem}
\begin{proof}
By using the conditions $\frac{\alpha_1+n}{p}=\alpha+n-\beta$ and $\alpha_1>-n$, it is not hard to check that
\begin{align*}
|\mathrm{H}^*_{\beta}f(x)|&\leq\frac{1}{\nu_{n}^{1-\frac{\beta}{n}}}\int_{|y|\geq|x|}|y|^{-\frac{n+\alpha_1}{p}}|f(y)||y|^{\alpha}dy\\
&\leq|x|^{-\frac{n+\alpha_1}{p}}\frac{1}{\nu_{n}^{1-\frac{\beta}{n}}}\int_{|y|\geq|x|}|f(y)||y|^{\alpha}dy\\
&\leq|x|^{-\frac{n+\alpha_1}{p}}\frac{\omega_{n}^{\frac{1}{\tilde{p_1}'}}}{\nu_{n}^{1-\frac{\beta}{n}}}\int_{|x|}^{\infty}\Big(\int_{{S}^{n-1}}|f(\rho\theta)|^{\tilde{p_1}}d\sigma(\theta)\Big)^{1/\tilde{p_1}}\rho^{\alpha}\rho^{n-1}d\rho\\
&\leq|x|^{-\frac{n+\alpha_1}{p}}\frac{\omega_{n}^{\frac{1}{\tilde{p_1}'}}}{\nu_{n}^{1-\frac{\beta}{n}}}\|f\|_{L^{1}_{rad}L^{\tilde{p_1}}_{ang}(\mathbb R^{n},|x|^{\alpha})}\\
&=:Q_0|x|^{-\frac{n+\alpha_1}{p}}.
\end{align*}
 Then, for $\lambda>0$,  we have
 $$\{x\in\mathbb R^{n}: |\mathrm{H}^{*}_{\beta}(f)(x)|>\lambda\}\subset\{x\in\mathbb R^{n}:: Q_0|x|^{-\frac{n+\alpha_1}{p}}>\lambda\}.$$
Thus, we get
 \begin{align*}
&\sup\limits_{\lambda>0}\lambda\big\|\chi_{\{x\in\mathbb R^n:  |\mathrm{H}^{*}_{\beta}f(x)|>\lambda\}}\big\|_{L^{p}_{rad}L^{\tilde{p_2}}_{ang}(\mathbb R^n,|x|^{\alpha_1})}\\
&\leq\sup\limits_{\lambda>0}\lambda\Bigg\|\chi_{\{x\in\mathbb R^n: |x|<\Big(\frac{Q_0}{\lambda}\Big)^{\frac{p}{n+\alpha_1}}\}}\Bigg\|_{L^{p}_{rad}L^{\tilde{p_2}}_{ang}(\mathbb R^n,|x|^{\alpha_1})}\\
&=\sup\limits_{\lambda>0}\lambda{{\omega}_{n}}^{1/\tilde{p_2}}\Bigg(\int_{0}^{\big(\frac{Q_0}{\lambda}\big)^{\frac{p}{n+\alpha_1}}}r^{\alpha_1+n-1}dr\Bigg)^{1/p}\\
&={{\omega}_{n}}^{1/\tilde{p_2}}\Big(\frac{1}{n+\alpha_1}\Big)^{1/p}Q_0\\
&=\frac{\omega_{n}^{\frac{1}{\tilde{p_1}'}+\frac{1}{\tilde{p_2}}}}{\nu_{n}^{1-\frac{\beta}{n}}}\Big(\frac{1}{n+\alpha_1}\Big)^{1/p}\|f\|_{L^{1}_{rad}L^{\tilde{p_1}}_{ang}(\mathbb R^{n},|x|^{\alpha})}.
\end{align*}
On the other hand, we  choose the function
$$f_{\varepsilon}(x)=|x|^{-(\beta+n+\alpha)/\varepsilon}\chi_{\{ x\in\mathbb R^n: |x|>1\}}(x),\quad 0<\varepsilon<1,$$
and
$$\|f_{\varepsilon}\|_{L^{1}_{rad}L^{\tilde{p_1}}_{ang}(\mathbb R^{n},|x|^{\alpha})}={{\omega}_{n}}^{1/\tilde{p_1}}\frac{1}{\frac{\beta+n+\alpha}{\varepsilon}-n-\alpha}.$$
If $\alpha>-n$ and $0<\varepsilon<1$, then for $|x|>1$, we have
$$\mathrm{H}^{*}_{\beta}(f_{\varepsilon})(x)=\frac{\omega_{n}}{\nu_{n}^{1-\frac{\beta}{n}}}\frac{|x|^{\beta-(\frac{\beta+n+\alpha}{\varepsilon})}}{\frac{\beta+n+\alpha}{\varepsilon}-\beta}:=R_1|x|^{\beta-(\frac{\beta+n+\alpha}{\varepsilon})}.$$
If $0\leq\beta<n$ and $0<\varepsilon<1$,  we also have
$$\mathrm{H}^{*}_{\beta}(f_{\varepsilon})(x)=\frac{\omega_{n}}{\nu_{n}^{1-\frac{\beta}{n}}}\frac{1}{\frac{\beta+n+\alpha}{\varepsilon}-\beta}:=R_2.$$
This implies that
$$\{x\in\mathbb R^{n}: |\mathrm{H}^{*}_{\beta}(f_{\varepsilon})(x)|>\lambda\}=\{|x|<1: R_2>\lambda\}\cup\{|x|\geq1: R_1|x|^{\beta-(\frac{\beta+n+\alpha}{\varepsilon})}>\lambda\}.$$
If $\lambda\geq R_1$, then
$$\Big\{x\in\mathbb R^{n}: |\mathrm{H}^{*}_{\beta}(f_{\varepsilon})(x)|>\lambda\Big\}=\emptyset.$$
Therefore, for  $0\leq\beta<n$ and $0<\varepsilon<1$, we have

\begin{align*}
&\sup\limits_{0<\lambda<R_1}\lambda\big\|\chi_{\{x\in\mathbb R^n:  |\mathrm{H}^{*}_{\beta}f(x)|>\lambda\}}\big\|_{L^{p}_{rad}L^{\tilde{p_2}}_{ang}(\mathbb R^n,|x|^{\alpha})}\\
&={{\omega}_{n}}^{1/\tilde{p_2}}\Big(\frac{1}{n+\alpha_1}\Big)^{1/p}R_1^{\frac{n+\alpha-\beta}{\frac{\beta+n+\alpha}{\varepsilon}-\beta}}\sup\limits_{0<\lambda<R_1}\lambda^{1-\frac{n+\alpha-\beta}{\frac{\beta+n+\alpha}{\varepsilon}-\beta}}\\
&={{\omega}_{n}}^{1/\tilde{p_2}}\Big(\frac{1}{n+\alpha_1}\Big)^{1/p}R_1\\
&={{\omega}_{n}}^{1/\tilde{p_2}}\Big(\frac{1}{n+\alpha_1}\Big)^{1/p}\frac{\omega_{n}}{\nu_{n}^{1-\frac{\beta}{n}}}\frac{1}{\frac{\beta+n+\alpha}{\varepsilon}-\beta}\\
&=\frac{\omega_{n}^{\frac{1}{\tilde{p_1}'}+\frac{1}{\tilde{p_2}}}}{\nu_{n}^{1-\frac{\beta}{n}}}\Big(\frac{1}{n+\alpha_1}\Big)^{1/p}\frac{\frac{\beta+n+\alpha}{\varepsilon}-n-\alpha}{\frac{\beta+n+\alpha}{\varepsilon}-\beta}\|f_{\varepsilon}\|_{L^{1}_{rad}L^{\tilde{p_1}}_{ang}(\mathbb R^{n},|x|^{\alpha})}.
\end{align*}
Letting $\varepsilon\rightarrow 0^+$, we have
$$\frac{\frac{\beta+n+\alpha}{\varepsilon}-n-\alpha}{\frac{\beta+n+\alpha}{\varepsilon}-\beta}=1.$$
Therefore, we conclude that
$$\|\mathrm{H}^*_{\beta}\|_{L^{1}_{rad}L^{\tilde{p_1}}_{ang}(\mathbb R^{n},|x|^{\alpha})\rightarrow \mathcal{W}L^{p}_{rad}L^{\tilde{p_2}}_{ang}(\mathbb R^{n},|x|^{\alpha_2})}=\frac{\omega_{n}^{\frac{1}{\tilde{p_1}'}+\frac{1}{\tilde{p_2}}}}{\nu_{n}^{1-\frac{\beta}{n}}}\Big(\frac{1}{n+\alpha_1}\Big)^{1/p}.$$
The proof of Theorem \ref{th4.5} is completed.
\end{proof}

\section{Sharp Bounds of weighted Hardy-Littlewood averages}

In \cite{LT}, the first author and third author obtained the following the mixed radial-angular estimates for the weighted Hardy-Littlewood average $U_{\psi}(f)$ and the weighted Ces\`{a}ro average $V_{\psi}(f)$, respectively.

{\bf Theorem P}
Let $\psi:[0,1]\rightarrow[0,\infty)$ be a function and let $1<p<\infty$, $1<p_2\leq p_1<\infty$, then we have

(i) $U_{\psi}(f): L^{p}_{rad}L^{p_1}_{ang}(\mathbb R^{n})\rightarrow L^{p}_{rad}L^{p_2}_{ang}(\mathbb R^{n})$ exists as a bounded operator if and only if
$$\int_{0}^{1}t^{-n/p}\psi(t)dt<\infty.$$

Moreover, the operator norm of $U_{\psi}$  is given by
$$\|U_{\psi}\|_{L^{p}_{rad}L^{p_2}_{ang}(\mathbb R^{n})\rightarrow L^{p}_{rad}L^{p_1}_{ang}(\mathbb R^{n})}={{\omega}_{n}}^{1/p_2-1/p_1}\int_{0}^{1}t^{-n/p}\psi(t)dt.$$

(ii) $V_{\psi}(f):  L^{p}_{rad}L^{p_1}_{ang}(\mathbb R^{n})\rightarrow L^{p}_{rad}L^{p_2}_{ang}(\mathbb R^{n})$ exists as a bounded operator if and only if
$$\int_{0}^{1}t^{-n(1-1/p)}\psi(t)dt<\infty.$$

Moreover, the operator norm of $V_{\psi}$  is given by
$$\|V_{\psi}\|_{L^{p}_{rad}L^{p_2}_{ang}(\mathbb R^{n})\rightarrow L^{p}_{rad}L^{p_1}_{ang}(\mathbb R^{n})}={{\omega}_{n}}^{1/p_2-1/p_1}\int_{0}^{1}t^{-n(1-1/p)}\psi(t)dt.$$

Next, we will extended Theorems G and H to the mixed radial-angular homogeneous Herz spaces and Morrey-Herz spaces, respectively.
\begin{theorem}\label{th5.2}
Let $\psi:[0,1]\rightarrow[0,\infty)$ be a function and let $\alpha\in\mathbb R$, $1<p,q,\tilde{p}<\infty$, then we have

(i) $U_{\psi}(f): \dot{K}^{\alpha,q}_{L^{p}_{\rm rad}L^{\tilde{p}}_{\rm ang}}(\mathbb R^{n})\rightarrow \dot{K}^{\alpha,q}_{L^{p}_{\rm rad}L^{\tilde{p}}_{\rm ang}}(\mathbb R^{n})$ exists as a bounded operator if
$$\int_{0}^{1}t^{-\alpha-n/p}\psi(t)dt<\infty.$$

Moreover, the operator norm of $U_{\psi}$  satisfies
$$\|U_{\psi}\|_{\dot{K}^{\alpha,q}_{L^{p}_{\rm rad}L^{\tilde{p}}_{\rm ang}}(\mathbb R^{n})\rightarrow \dot{K}^{\alpha,q}_{L^{p}_{\rm rad}L^{\tilde{p}}_{\rm ang}}(\mathbb R^{n})}\simeq\int_{0}^{1}t^{-\alpha-n/p}\psi(t)dt.$$

(ii) $V_{\psi}(f): \dot{K}^{\alpha,q}_{L^{p}_{\rm rad}L^{\tilde{p}}_{\rm ang}}(\mathbb R^{n})\rightarrow \dot{K}^{\alpha,q}_{L^{p}_{\rm rad}L^{\tilde{p}}_{\rm ang}}(\mathbb R^{n})$  exists as a bounded operator if
$$\int_{0}^{1}t^{\alpha-n(1-1/p)}\psi(t)dt<\infty.$$

Moreover, the operator norm of $V_{\psi}$   satisfies
$$\|V_{\psi}\|_{\dot{K}^{\alpha,q}_{L^{p}_{\rm rad}L^{\tilde{p}}_{\rm ang}}(\mathbb R^{n})\rightarrow \dot{K}^{\alpha,q}_{L^{p}_{\rm rad}L^{\tilde{p}}_{\rm ang}}(\mathbb R^{n})}\simeq\int_{0}^{1}t^{\alpha-n(1-1/p)}\psi(t)dt.$$
\end{theorem}

\begin{theorem}\label{th5.3}
Let $\psi:[0,1]\rightarrow[0,\infty)$ be a function and let $\alpha\in\mathbb R, \lambda>0$, $1<p,q,\tilde{p}<\infty$, then we have

(i) $U_{\psi}(f):{M\dot{K}}^{\alpha,q,\lambda}_{L^{p}_{\rm rad}L^{\tilde{p}}_{\rm ang}}(\mathbb R^{n})\rightarrow {M\dot{K}}^{\alpha,q,\lambda}_{L^{p}_{\rm rad}L^{\tilde{p}}_{\rm ang}}(\mathbb R^{n})$ exists as a bounded operator if
$$\int_{0}^{1}t^{-\alpha-n/p+\lambda}\psi(t)dt<\infty.$$

Moreover, the operator norm of $U_{\psi}$  satisfies
$$\|U_{\psi}\|_{{M\dot{K}}^{\alpha,q,\lambda}_{L^{p}_{\rm rad}L^{\tilde{p}}_{\rm ang}}(\mathbb R^{n})\rightarrow {M\dot{K}}^{\alpha,q,\lambda}_{L^{p}_{\rm rad}L^{\tilde{p}}_{\rm ang}}(\mathbb R^{n})}\simeq\int_{0}^{1}t^{-\alpha-n/p+\lambda}\psi(t)dt.$$

(ii) $V_{\psi}(f): {M\dot{K}}^{\alpha,q,\lambda}_{L^{p}_{\rm rad}L^{\tilde{p}}_{\rm ang}}(\mathbb R^{n})\rightarrow {M\dot{K}}^{\alpha,q,\lambda}_{L^{p}_{\rm rad}L^{\tilde{p}}_{\rm ang}}(\mathbb R^{n})$  exists as a bounded operator if
$$\int_{0}^{1}t^{\alpha-\lambda-n(1-1/p)}\psi(t)dt<\infty.$$

Moreover, the operator norm of $V_{\psi}$   satisfies
$$\|V_{\psi}\|_{{M\dot{K}}^{\alpha,q,\lambda}_{L^{p}_{\rm rad}L^{\tilde{p}}_{\rm ang}}(\mathbb R^{n})\rightarrow {M\dot{K}}^{\alpha,q,\lambda}_{L^{p}_{\rm rad}L^{\tilde{p}}_{\rm ang}}(\mathbb R^{n})}\simeq\int_{0}^{1}t^{\alpha-\lambda-n(1-1/p)}\psi(t)dt.$$
\end{theorem}

\begin{proof}  The proofs of Theorems \ref{th5.2} and \ref{th5.3} are similar, thus we only need to prove Theorem \ref{th5.3}.

(i) Suppose $\int_{0}^{1}t^{-\alpha-n/p+\lambda}\psi(t)dt<\infty$, then
\begin{align*}
&\|U_{\psi}(f)\chi_k\|_{L^p_{\rm rad}L^{\tilde{p}}_{\rm ang}(\mathbb R^{n})}\\
&=\Big(\int_{2^{k-1}}^{2^k}\Big(\int_{{S}^{n-1}}|U_{\psi}(f)(r\theta)|^{\tilde{p}}d\sigma(\theta)\Big)^{{p}/\tilde{p}}r^{n-1}dr\Big)^{1/p}\\
&=\Big(\int_{2^{k-1}}^{2^k}\Big(\int_{{S}^{n-1}}|\int_{0}^{1}f(tr\theta)\psi(t)dt|^{\tilde{p}}d\sigma(\theta)\Big)^{{p}/\tilde{p}}r^{n-1}dr\Big)^{1/p}\\
&\leq\Big(\int_{2^{k-1}}^{2^k}\Big(\int_{0}^{1}\Big(\int_{{S}^{n-1}}|f(tr\theta)|^{\tilde{p}}d\sigma(\theta)\Big)^{1/\tilde{p}}\psi(t)dt\Big)^{{p}}r^{n-1}dr\Big)^{1/p}\\
&\leq\int_{0}^{1}\Big(\int_{2^{k-1}}^{2^k}\Big(\int_{{S}^{n-1}}|f(tr\theta)|^{\tilde{p}}d\sigma(\theta)\Big)^{p/\tilde{p}}r^{n-1}dr\Big)^{1/p}\psi(t)dt\\
&=\int_{0}^{1}\Big(\int_{2^{k-1}t}^{2^kt}\Big(\int_{{S}^{n-1}}|f(r\theta)|^{\tilde{p}}d\sigma(\theta)\Big)^{p/\tilde{p}}r^{n-1}dr\Big)^{1/p}t^{-n/p}\psi(t)dt.
\end{align*}
Then for any $0<t<1$, there is $m\in\mathbb{Z}$, such that $2^{m-1}<t\leq2^{m}$, by the Minkowski inequality, we have
\begin{align*}
&\|U_{\psi}(f)\chi_k\|_{L^p_{\rm rad}L^{\tilde{p}}_{\rm ang}(\mathbb R^{n})}\\
&\leq\int_{0}^{1}\Big(\int_{2^{k-m-2}t}^{2^{k+m}t}\Big(\int_{{S}^{n-1}}|f(r\theta)|^{\tilde{p}}d\sigma(\theta)\Big)^{p/\tilde{p}}r^{n-1}dr\Big)^{1/p}t^{-n/p}\psi(t)dt\\
&\leq\int_{0}^{1}\big(\|f\chi_{k+m-1}\|_{L^{p}_{rad}L^{\tilde{p}}_{ang}(\mathbb R^{n})}+\|f\chi_{k+m}\|_{L^{p}_{rad}L^{\tilde{p}}_{ang}(\mathbb R^{n})\big)}t^{-n/p}\psi(t)dt.
\end{align*}
Therefore, we have
\begin{align*}
&\|U_{\psi}(f)\|_{{M\dot{K}}^{\alpha,q,\lambda}_{L^{p}_{\rm rad}L^{\tilde{p}}_{\rm ang}}(\mathbb R^{n})}\\
&=\sup\limits_{k_0\in\mathbb{Z}}2^{-k_0\lambda}\Big\{\sum\limits_{k=-\infty}^{k_0}2^{k\alpha q} \|U_{\psi}(f)\chi_k\|^q_{L^{p}_{\rm rad}L^{\tilde{p}}_{\rm ang}(\mathbb R^{n})}\Big\}^{1/q}\\
&\leq\sup\limits_{k_0\in\mathbb{Z}}2^{-k_0\lambda}\Big\{\sum\limits_{k=-\infty}^{k_0}2^{k\alpha q} \Big(\int_{0}^{1}\big(\|f\chi_{k+m-1}\|_{L^{p}_{rad}L^{\tilde{p}}_{ang}(\mathbb R^{n})}\\
&\quad+\|f\chi_{k+m}\|_{L^{p}_{rad}L^{\tilde{p}}_{ang}(\mathbb R^{n})\big)}t^{-n/p}\psi(t)dt\Big)\Big\}^{1/q}\\
&\leq\sup\limits_{k_0\in\mathbb{Z}}2^{-k_0\lambda}\int_{0}^{1}\Big(\sum\limits_{k=-\infty}^{k_0}2^{k\alpha q}\|f\chi_{k+m-1}\|^q_{L^{p}_{rad}L^{\tilde{p}}_{ang}(\mathbb R^{n})}\Big)^{1/q}t^{-n/p}\psi(t)dt\\
&\quad+\sup\limits_{k_0\in\mathbb{Z}}2^{-k_0\lambda}\int_{0}^{1}\Big(\sum\limits_{k=-\infty}^{k_0}2^{k\alpha q}\|f\chi_{k+m}\|^q_{L^{p}_{rad}L^{\tilde{p}}_{ang}(\mathbb R^{n})}\Big)^{1/q}t^{-n/p}\psi(t)dt\\
&\lesssim\|f\|_{{M\dot{K}}^{\alpha,q,\lambda}_{L^{p}_{\rm rad}L^{\tilde{p}}_{\rm ang}}(\mathbb R^{n})}\int_{0}^{1}t^{-\alpha-n/p+\lambda}\psi(t)dt.
\end{align*}

On the other hand, we consider the following function
$$f(x)=|x|^{-\alpha-n/p+\lambda}.$$
Then,
$$U_{\psi}(f)(x)=f(x)\int_{0}^{1}t^{-\alpha-n/p+\lambda}\psi(t).$$
If $\alpha\neq\lambda$, then
\begin{align*}
&\|f\chi_k\|^p_{L^p_{\rm rad}L^{\tilde{p}}_{\rm ang}(\mathbb R^{n})}\\
&=\omega_n^{p/\tilde{p}}\int_{2^{k-1}}^{2^k}r^{(-\alpha-n/p+\lambda)p}r^{n-1}dr\\
&\simeq1.
\end{align*}
If $\alpha=\lambda$, then
$$\|f\chi_k\|^p_{L^p_{\rm rad}L^{\tilde{p}}_{\rm ang}(\mathbb R^{n})}=\omega_n^{p/\tilde{p}}\ln2\simeq1.$$
Therefore, we derive that
$$\|U_{\psi}(f)\|_{{M\dot{K}}^{\alpha,q,\lambda}_{L^{p}_{\rm rad}L^{\tilde{p}}_{\rm ang}}(\mathbb R^{n})}\simeq\|f\|_{{M\dot{K}}^{\alpha,q,\lambda}_{L^{p}_{\rm rad}L^{\tilde{p}}_{\rm ang}}(\mathbb R^{n})}\int_{0}^{1}t^{-\alpha-n/p+\lambda}\psi(t)dt.$$
The proof of Theorem \ref{th5.3} is completed.
\end{proof}

\end{document}